\documentclass{article}
\usepackage{graphicx}
\usepackage{authblk}
\usepackage{hyperref}
\hypersetup{
	pdftitle={Optimality conditions in control problems with random state constraints},
	pdfauthor={C. Geiersbach and R. Henrion}
}
\usepackage{amsmath, color, verbatim}
\usepackage[numbers,sort]{natbib}
\usepackage{amsfonts, mathtools, enumitem} 
\usepackage{hyperref}
\usepackage{todonotes}
\usepackage{enumitem}
\usepackage{titlesec}
\usepackage{cleveref}
\usepackage[left=2cm,right=2cm,top=2cm,bottom=2cm]{geometry}
\hypersetup{
  colorlinks=true,
}
\usepackage{amssymb}
\usepackage{amsthm}
\DeclareMathOperator*{\esssup}{ess\,sup}

\newcommand{\norm}[2]{\left\lVert #1\right\rVert_{#2}}

\newtheorem{theorem}{Theorem}[section]

\newtheorem{lemma}[theorem]{Lemma}
\newtheorem{proposition}[theorem]{Proposition}

\newtheorem{example}[theorem]{Example}
\newtheorem{corollary}[theorem]{Corollary}
\newtheorem{remark}[theorem]{Remark}
\newcounter{assumption}
\newtheorem{assumption}[theorem]{Assumption}

\newcommand{\E}{\mathbb{E}}
\newcommand{\pP}{\mathbb{P}}
\newcommand{\R}{\mathbb{R}}

\newcommand{\D}{\text{ d}}

\newcounter{subassumption}[assumption]
\renewcommand{\thesubassumption}{(\textit{\roman{subassumption}})}
\makeatletter
\renewcommand{\p@subassumption}{\theassumption}% Counter prefix.
\makeatother
\newcommand{\subasu}{% Just like \item in a list, but for an asu
  \refstepcounter{subassumption}%
  \thesubassumption~\ignorespaces}
 \crefname{subassumption}{Assumption}{Assumptions}

\begin{document}

\title{Optimality conditions in control problems with random state constraints in probabilistic or almost-sure form}

%\subtitle{}

\author{Caroline Geiersbach\thanks{Weierstrass Institute, 10117 Berlin, Germany 
  (\texttt{caroline.geiersbach@wias-berlin.de})}
\and Ren\'e Henrion\thanks{Weierstrass Institute, 10117 Berlin, Germany 
  (\texttt{rene.henrion@wias-berlin.de})}}

\date{\today}
%The correct dates will be entered by the editor.

\maketitle

\begin{abstract}
In this paper, we discuss optimality conditions for optimization problems {involving} random state constraints, which are modeled in probabilistic or almost sure form. While the latter can be understood as the limiting case of the former, the derivation of optimality conditions requires substantially different approaches. We apply them to a linear elliptic partial differential equation (PDE) with random inputs.
In the probabilistic case, we rely on the spherical-radial decomposition of Gaussian random vectors in order to formulate fully explicit optimality conditions involving a spherical integral. In the almost sure case, we derive optimality conditions and compare them to a model based on robust constraints with respect to the (compact) support of the given distribution.
\end{abstract}

\section{Introduction}
Uncertainty appears in many applications in optimization, from finance to medicine to engineering, due to unknown inputs or parameters in underlying systems. Ignoring uncertainty can lead to solutions that poorly reflect how the system behaves with high probability. This fact holds true no matter if the space of decisions is finite-dimensional or infinite-dimensional as in PDE-constrained optimization.  The present work concerns itself with the second context. For an introductory monograph on PDE-constrained optimization under uncertainty, we refer to \cite{martinez2018}. One way to classify different approaches in this field consists in considering the degree of risk involved in decision making: { on the one hand, one may deal with risk-neutral (expectation-based) models, e.g., \cite{conti2009,Gahururu2022,Milz2023}. On the other hand, one may be} rather interested in absolutely safe decisions and address uncertainty via almost sure constraints \cite{Geiersbach2022a,Geiersbach2021c}. In between these two extremes, reasonable compromises can be made based on several risk-averse approaches such as conditional value-at-risk (CVaR), value-at-risk (VaR), stochastic dominance, or robust constraints 
\cite{alphonse2022risk,conti2018,Kolvenbach2018,Kouri2018,martinez2018}.
Besides modeling, numerical solution, and structural analysis, there have been efforts in the past few years to derive optimality conditions for problems involving random PDEs, see \cite{Kouri2018} (without state constraints, meaning additional constraints on the solution of the PDE) or \cite{Gahururu2022,Geiersbach2022a,Geiersbach2021c} (with almost sure state constraints). 

In this paper, we consider optimality conditions for problems of the form
\begin{equation}
\label{eq:problem-introduction}
\min\limits_{u\in U}\,\,F(u)\quad\mbox{subject to (s.t.)}\quad\mathbb{P}(g(u,\xi )\leq 0)\geq p\quad (p\in (0,1]),
\end{equation}
where $U$ is a Banach space, $\xi$ is some random vector on a probability space $(\Omega ,\mathcal{F},\mathbb{P})$, $g$ is a (random) constraint function, and $p$ is some given probability level. 
Constraints like those appearing in \eqref{eq:problem-introduction} are referred to as probabilistic or chance constraints. They express the condition that a feasible decision $u$ has to guarantee the satisfaction of the random inequality constraint $g\leq 0$ at least with probability $p$. This can be equivalently formulated by a value-at-risk constraint as mentioned above. 
Probabilistic constraints have been introduced in the context of operations research problems more than 60 years ago by Charnes et al.~\cite{charnes1958}. Their algorithmical treatment and theoretical understanding have been fundamentally advanced by Prékopa, whose monograph \cite{prekopa1995} is still a fundamental reference on this topic in finite-dimensional optimization. Modern presentations can be found in \cite{Shapiro2009} or \cite{vanAckooij2020}, respectively. {The last two decades have seen substantial progress mainly in the development of numerical approaches to tackle probabilistic  constraints under discrete or continuous random distributions. We mention here as examples the methods of {\it scenario approximation} \cite{Campi2008,Esfahani2015}, {\it sample average approximation} \cite{Pagnoncelli2009}, {\it convex approximation} \cite{Nemirovski2006}, {\it difference-of-convex-functions (DC) approximation} \cite{Hong2011,ackooijdc}, {\it smooth approximation} \cite{pena,geletu17}, {\it kernel density estimators} \cite{caillau2018,Schuster2022},
{\it spherical-radial decomposition} \cite{vanAckooij_Henrion_2014,Farshbaf-Shaker2020}, {\it importance sampling} \cite{Barrera2016}, or {\it bilevel optimization} \cite{Tong}.
}
In recent years, there has been growing interest in considering probabilistic constraints in the framework of optimal control in general or PDE-constrained optimization in particular, e.g., \cite{caillau2018,Farshbaf-Shaker2020,Farshbaf-Shaker2018,geletu2020,goettlich2021,Kouri2023,Perez2022,Teka23}. These works include proposals for numerical approaches as well as structural investigations, e.g, existence and stability of solutions.

The abstract problem \eqref{eq:problem-introduction} encompasses PDE-constrained optimization subject to probabilistic state constraints. { The present work is, to the best of our knowledge, a first attempt to derive {\bf fully explicit} optimality conditions for a PDE-constrained optimization problem subject to {\bf joint} probabilistic random state constraints {\bf uniformly} with respect to the domain, considering the original problem {\bf as is}, taking into account the potential {\bf nonsmoothness} inherent to the model despite its possibly smooth input data and exploiting some hidden convexity in order to derive optimality conditions that are {\bf necessary and sufficient}. Another major contribution in this paper involves the comparison to the closely related problems involving {\bf almost sure and robust constraints}.}

{ We note that most work in the literature on PDE-constrained optimization subject to probabilistic constraints is not related to optimality conditions and does not consider the original problem but approximations thereof. 
The uniformity of random state constraints with respect to the domain leads to an {\bf infinite} number of smooth random inequalities. This complicates the analysis of the original problem considerably. A typical remedy in the literature is to either make recourse to approximations as mentioned before or to treat probabilistic constraints in an individual way (separately for each point in the domain). The latter approach is well-known to be problematic because high probabilities on an individual level may still result in low probabilities for the uniform level. In the present work, we investigate the original problem while keeping the joint character of probabilistic constraints. For such constraints, the name ``probust" (probabilistic/robust) was coined. Their theoretical analysis \cite{Farshbaf-Shaker2018,vanAckooij2020b} and algorithmic treatment \cite{Berthold2022} is still in its infancy and Section~\ref{sec:prob-constraints} of the present work makes a contribution to these investigations as well. A major challenge in the derivation of optimality conditions is the explicit characterization of the subdifferential (if not derivative) of the potentially nonsmooth probability function. The present paper provides a fully explicit and exact formula (not just an upper estimate) for this subdifferential, which is ready to use for numerical approaches.
Our workhorse for the derivation of optimality conditions will be the so-called {\em spherical-radial decomposition} of Gaussian random vectors (which actually generalizes to any elliptically-symmetric distribution). This tool has proven to be useful in many applications involving the numerical treatment of probabilistic constraints (e.g., \cite{Berthold2022,Farshbaf-Shaker2020}) as well as for the analytic derivation of (sub-)gradients of the decision-dependent probability function (e.g., \cite{Hantoute2019,vanAckooij_Henrion_2014,ackkoij-perez-2019}). As a typical outcome, these (sub-)gradients are represented as spherical integrals, which makes them---contrary to derivative formulae in which the domain of integration is decision-dependent---amenable to numerical approximation and theoretical manipulation.}

It is noteworthy that admitting a probability level of $p=1$ in \eqref{eq:problem-introduction} allows one (at least formally) to observe almost sure constraints of the type ``$g(u,\xi)\leq 0\,\,\mathbb{P}$-a.s." in the model. We shall see, however, that the derivation of optimality conditions in Section~\ref{sec:prob-constraints} requires a restriction of probability levels to values strictly smaller than one. It does not make sense either to consider optimality conditions in the limit when driving $p$ towards one because then the Lagrange multiplier tends to infinity; see Example~\ref{asvsprob}. Therefore, in order to arrive at optimality conditions for almost sure constraints, one has to rather use functional-analytic approaches originally developed for two-stage stochastic programming \cite{Rockafellar1976,Rockafellar1976b,Rockafellar1975,Rockafellar1976c} and further for problems in PDE-constrained optimization in \cite{Geiersbach2022a,Geiersbach2021c}. { For the almost sure problem, this involves working with Bochner spaces of the type $L^\infty_{\pP}(\Omega,X)$. We do not require reflexivity of $X$ as in \cite{Geiersbach2022a,Geiersbach2021c} and we present an efficient approach for deriving optimality conditions that greatly simplifies the analysis when compared to \cite{Geiersbach2022a,Geiersbach2021c}. In this approach, optimality conditions are first derived using standard arguments with Lagrange multipliers in $(L^\infty_{\pP}(\Omega,X))^*$ and then these multipliers are made more explicit using the usual Yosida--Hewitt-type decomposition.} Almost sure constraints are closely tied with robust constraints of the type ``$g(u,z)\leq 0$'' for all $z$ belonging to the support of the random vector. { We derive optimality conditions for the robust problem using the standard arguments in the PDE-constrained optimization literature and make a connection to semi-infinite programming.} This opens yet another perspective to the derivation of optimality conditions in case one is interested in absolute safety. In Section~\ref{sec:almost-sure}, we will shed some light on various approaches to optimality conditions in such models and their comparison in the context of the PDE mentioned above. 

\subsection{Notation}
\label{subsec:notation}
The dual space of a Banach space $X$ is denoted by $X^*$. Throughout, the duality pairing $\langle \cdot, \cdot \rangle_{X^*,X}$ will be written as $\langle \cdot, \cdot \rangle$ since the underlying spaces can be easily guessed by the elements used in the pairing. Weak convergence is denoted by $\rightharpoonup$. A norm on a Banach space $X$ is denoted by $\lVert \cdot \rVert_X$ and the Euclidean norm is denoted by $\lVert \cdot \rVert_2$. 

Given a set $D \subset \R^d$, the space $L^p(D)$ denotes the usual Lebesgue space with the Lebesgue measure on $D$. For $q \in [1,\infty]$, the Sobolev space $W^{1,q}(D)$ is given by the set of $L^q(D)$ functions with weak derivatives in $L^q(D)$; the set $W_0^{1,q}(D)$ is the set of $W^{1,q}(D)$ functions that vanish on the boundary $\partial D$. We have $H^{1}(D)=W^{1,2}(D)$ and the dual space $W^{-1,q}(D):=(W^{1,q'}(D))^*$ with $\tfrac{1}{q}+\tfrac{1}{q'}=1.$

Given a probability space  $(\Omega, \mathcal{F}, \pP)$ and a Banach space $X$, the Bochner space $L_{\pP}^r(\Omega, X):=L^r(\Omega,\mathcal{F}, \pP; X)$ is the set of all (equivalence classes of) strongly measurable functions $y\colon\Omega \rightarrow X$ having finite norm, where the norm is given by
\begin{equation*}
\lVert y \rVert_{L_{\pP}^r(\Omega,X)}:= \begin{cases}                                     (\int_\Omega \lVert y(\omega) \rVert_X^r \D \pP(\omega))^{1/r}, \quad &r < \infty\\
                                     \esssup_{\omega \in \Omega} \lVert y(\omega) \rVert_X, \quad &r=\infty
                                    \end{cases}.
\end{equation*}
%An $X$-valued random variable $x$ is Bochner integrable if there exists a
%sequence $\{ x_n\}$ of $\pP$-simple functions $x_n\colon \Omega \rightarrow X$
%such that $\lim_{n \rightarrow \infty} \int_{\Omega} \lVert x_n(\omega)-x(\omega) \rVert_X \D \pP(\omega) = 0$.
%The limit of the integrals of $x_n$ gives the Bochner integral
%(the expectation), i.e., $
 % \E[x]:=\int_\Omega x(\omega) \D \pP(\omega) = \lim_{n \rightarrow \infty} \int_{\Omega} x_n(\omega) \D \pP(\omega).$
%This expectation is an element of $X$. 
Given a compact topological space $K$ and a Banach space $X$, the set $\mathcal{C}(K,X)$ is the set of continuous functions from $K$ into $X$, where the norm is given by $\norm{f}{\mathcal{C}(K,X)}:=\sup_{z \in K} \norm{f(z)}{X}.$
If $X = \R$, we write $\mathcal{C}(K)$ as a shorthand for $\mathcal{C}(K,\R).$ The set $\mathcal{C}^{0,\alpha}(K)$ denotes the space of $\alpha$-H\"older continuous functions on $D$ for $\alpha \in (0,1).$

Given Banach spaces $Y$ and $Z$, the set of bounded linear operators from $Y$ to $Z$ is denoted by 
$\mathcal{L}(Y,Z).$ The Fr\'echet derivative of $F\colon X \rightarrow \R$ is denoted by $DF \colon X \rightarrow \mathcal{L}(X,\R)=X^*$. For appropriate functions $f$ defined on $X$, we denote by $\partial f, \partial^C f, \partial^F f, \partial^M f$ the classical subdifferential of convex analysis, the Clarke subdifferential, the Fr\'echet subdifferential, and the Mordukhovich subdifferential, respectively. For the definition of these objects, we refer the reader to the standard monographs 
\cite{clarke1983,mordukhovich2006,peypouquet2015convex}. Partial subdifferentials are understood as subdifferentials of partial functions.  For a locally Lipschitzian function $f:X\to\mathbb{R}$, we denote by
\[
f^0(x;h):=\limsup\limits_{x'\to x, t\downarrow 0}\frac{f(x'+th)-f(x')}{h}
\]
its Clarke directional derivative at $x\in X$ in direction $h\in X$. We recall that a locally Lipschitzian function is called Clarke regular at some point if for all directions its ordinary directional derivative exists at that point and coincides with its Clarke directional derivative.
\subsection{A Model PDE}
\label{subsec:PDE}
As mentioned in the introduction, we will focus on the derivation of optimality conditions for a problem involving a random linear elliptic PDE. In this section, we will introduce this PDE and provide regularity results that will be of use later. We first consider the following parametrized elliptic PDE
\begin{equation}
 \label{eq:robust-PDE}
    \begin{alignedat}{3}
  \quad -\Delta  \hat{y}(x,z)  &=  u(x) + f(x,z), & &&\quad (x,z) \in D \times \Xi, \\
 \hat{y}(x,z) &=0, & &&\quad (x,z) \in \partial D \times \Xi.
    \end{alignedat}
\end{equation}
Here, $D$ is an open bounded subset of $\R^d$ with the boundary $\partial D$ and $\Xi$ is a (possibly unbounded) subset of $\R^m$ { playing the role of the support of a probability distribution. In Section \eqref{sec:prob-constraints} on probabilistic constraints, $\Xi$ will be the whole space, whereas in Section  \eqref{sec:almost-sure} on almost-sure and robust constraints it is restricted to be a compact set.} The Laplacian $\Delta y = \sum_{i=1}^d \frac{\partial^2}{\partial x_i^2} y$ acts only on the variable $x.$ For the first result, we will require the domain to be of \textit{class $S$}, meaning that there exist constants $\gamma \in (0,1)$ and $r_0 >0$ such that $\text{meas}(B_{r}(x)\backslash D) \geq \gamma \text{meas}( B_r(x))$ for all $x \in \partial D$ and for all $ r< r_0$. This classical definition from \cite{Kinderlehrer1980} prevents inward cusps of $D$, and is a mild requirement that covers many domains of interest from the literature, such as Lipschitz and convex domains.
We impose the following assumptions on problem \eqref{eq:robust-PDE}.
\begin{assumption}
\label{ass:PDE-standing}
The open and bounded set $D\subseteq\mathbb{R}^d$ ($d=1,2,3$) is of class $S$. Additionally, $u \in L^2(D)$ and the function $f\colon \mathbb{R}^d\times\mathbb{R}^m\to\mathbb{R}$ is defined by 
\begin{equation}
\label{eq:f}
 f(x,z):= f_0(x) + \sum_{i=1}^m z_i\phi_i(x)
\end{equation}
for some given $f_0, \phi_i\in L^2(D)$. 
\end{assumption}
The choice \eqref{eq:f} will be motivated later in Remark \ref{rem:about-the-PDE}. First, we establish well-posedness of \eqref{eq:robust-PDE} under this assumption.

\begin{lemma}
\label{lem:cont-dep-sol}
Suppose Assumption~\ref{ass:PDE-standing} holds. Then for any $u, f_0, \phi_i\in L^2(D)$ ($i=1, \dots, m$) and every $z \in \Xi$, there exists a unique solution $\hat{y}(\cdot,z) \in H_0^{1}(D)\cap \mathcal{C}(\bar{D})$ of \eqref{eq:robust-PDE}, { where $\bar{D}$ denotes the closure of $D$}. Moreover, there exists a constant $C>0$ such that 
\begin{equation}
\label{eq:apriori-PDE-first}
     \norm{\hat{y}(\cdot,z)}{\mathcal{C}(\bar{D})}   \leq C(\norm{z}{2} + \norm{u+f_0}{L^2(D)}).
\end{equation}
\end{lemma}

\begin{proof}
By the Lax--Milgram lemma, the operator $-\Delta$ defines an isomorphism between $H_0^1(D)$ and $H^{-1}(D).$ In particular, problem \eqref{eq:robust-PDE} has a unique solution $\hat{y}(\cdot,z) \in H_0^{1}(D)$ for every $z \in \Xi$ and every $u, f_0, \phi_i \in L^2(D)$ ($i=1, \dots, m$). 
%Moreover, there exists a constant $c_1>0$ depending only on the geometry and dimension of the domain such that 
%\begin{equation}
%\label{eq:H2-estimate}
%\lVert y(\cdot,\omega)\rVert_{H^1(D)} \leq  c_1 \lVert u+f(\cdot,\xi(\omega))\rVert_{L^2(D)}.
%\end{equation}
For $d=1$, \eqref{eq:apriori-PDE-first} follows trivially. For $d=2, 3$, { continuity of the solutions and the estimate \eqref{eq:apriori-PDE-first}} follows by \cite[Chapter II, Appendix B and C]{Kinderlehrer1980}.
%: the solution $y(\cdot,\omega)$ is H\"older continuous
%; more precisely, there exists $c_2>0$ and $\alpha \in (0,1)$ such that
%\[
%\sup_{x,\hat{x} \in B_r(z)} |y(x,\omega)-y(\hat{x},\omega)| \leq c_2 \norm{u+f(\cdot,\xi(\omega))}{L^2(D)}r^\alpha
%\]
%for all $z \in \R^d$ and $r>0.$
% %Since $H^2(D)$ is continuously embedded into $\mathcal{C}(\Bar{D})$ (for $d=1,2,3$), we have
% \[\norm{y(\cdot,\omega)}{\mathcal{C}(\Bar{D})}  \leq c_D c_1 \lVert u+f(\cdot,\xi(\omega))\rVert_{L^2(D)},\]
% where $c_D$ is the embedding constant from $H^2(D) \hookrightarrow \mathcal{C}(\Bar{D})$. 
% Using \eqref{eq:f}, we get \eqref{eq:apriori-PDE} for a large enough $C$.
\end{proof}

It is worth mentioning that \cite{Kinderlehrer1980} gives an even stronger result than we need here, namely H\"older continuity of solutions as well as continuity of the operator $A_q^{-1}\colon W^{-1,q}(D)\rightarrow \mathcal{C}^{0,\alpha}(\bar{D})$ for some $\alpha \in (0,1)$, $A_q$ being the part of $-\Delta$ in $W^{-1,q}(D)$ ($q=4$). Such regularity can be obtained for more general elliptic PDEs than presented here; see \cite{Haller2022a} and the references therein. Due to the continuity of the embeddings $\iota_1\colon L^2(D) \rightarrow W^{-1,q}(D)$ (for $d=1,2,3$) and $\iota_2\colon \mathcal{C}^{0,\alpha}(\bar{D}) \rightarrow \mathcal{C}(\bar{D})$, the operator $A \colon L^2(D) \rightarrow \mathcal{C}(\bar{D})$ defined by $A:=\iota_2 \circ A_q \circ \iota_1$ has a continuous inverse. With this operator, we can define the parametrized control-to-state operator $S\colon L^2(D) \times \R^m \rightarrow \mathcal{C}(\bar{D})$ by
\begin{equation}
\label{affinedecomp}
 S(u,z):= A^{-1}(u+f(\cdot, z))=:P(u,z)+y_0.
\end{equation} 
Here, $P(u,z)=A^{-1}u+\sum_{i=1}^m z_i A^{-1}\phi_i$ and $y_0= A^{-1}f_0.$ 
Due to \eqref{eq:apriori-PDE-first}, we have
\begin{equation}
\label{eq:apriori-PDE}
     \norm{S(u,z)}{\mathcal{C}(\bar{D})}   \leq C(\norm{z}{2} + \norm{u+f_0}{L^2(D)}),
\end{equation}
making $S$ a continuous operator. Applying \eqref{eq:apriori-PDE} to the special case $f_0=0$, we also have 
\begin{equation}\label{pest}
\norm{P(u,z)}{\mathcal{C}(\bar{D})}   \leq C(\norm{z}{2} + \norm{u}{L^2(D)})  
\end{equation}
so that $P$ is a continuous linear operator.

Now, suppose $\xi=(\xi_1, \dots, \xi_m)\colon \Omega \rightarrow \R^m$ is a random vector. We will consider the following elliptic PDE with a random right-hand side generated by $\xi$:
\begin{equation}
    \begin{alignedat}{2}
 -\Delta  y(x,\omega)  &=  u(x) + f(x,\xi(\omega)), \quad &&x \in D \quad \text{ $\mathbb{P}$-a.s.}, \\
 y(x,\omega) &=0, \quad &&x \in \partial D \quad \text{$\mathbb{P}$-a.s.}
    \end{alignedat}
    \label{eq:PDE}
\end{equation}

\noindent
\begin{remark}
\label{rem:about-the-PDE}
The random source term from \eqref{eq:PDE} in combination with the structure  given by \eqref{eq:f} might result, for instance, from the truncation of a Karhunen--Lo\`eve expansion of a random field on $D$. Since the random vector $\xi$ has images in $\R^m$, the source term also satisfies the finite-dimensional noise assumption frequently employed in numerical simulations. 
It is worth mentioning that, for the purpose of this work, we exclude the case of random diffusion terms and that this, in combination with the structure afforded by \eqref{eq:f}, will be needed to prove convexity of the probability function in Section~\ref{subsec:prob-PDE}. 
\end{remark}

Finally, the following measurability statement holds true and thus justifies the setting of \eqref{eq:PDE}.% and the statements of Lemma \ref{lem:cont-dep-sol}.
\begin{lemma}\label{measstat}
For each $x\in\bar{D}$ the mapping $\omega\mapsto y(x,\omega )$ is measurable. %Moreover, the mapping $\omega\mapsto y(\cdot,\omega )$ is measurable.
\end{lemma}
\begin{proof}
Thanks to $y(x,\omega )=[S(u,\xi (\omega ))](x)$ and to $\xi$ being measurable as a random vector, the claim follows by the continuity of the mappings $z\mapsto [S(u,z)](x)$ for each $x\in\bar{D}$. To see this, let $z_n\to z$ and observe that, by linearity of $P$ and by \eqref{pest},
\[
|[S(u,z)](x)-[S(u,z_n)](x)|=|[P(0,z-z_n)](x)|\leq C\norm{z-z_n}{2}\to 0\quad\forall x\in\bar{D}.
\]
%Similarly, the continuity of the operator $S$ yields the measurability of the mapping $\omega\mapsto y(\cdot,\omega )=S(u,\xi (\omega ))$
%\RH{Ist das korrekt?} \CG{Ich glaube schon, dieses Argument folgt z.B. mit \cite[Lemma 8.2.3]{Aubin1990}.}
%\CG{CG to think about: whether pointwise (in $x$) measurability follows from this. Discussion re: Borel measurability vs. Aubin/Frankowska.}
\end{proof}

\section{Probabilistic constraints in optimal control}
\label{sec:prob-constraints}
Before dealing with probabilistic state constraints in the application, we collect some general results on probabilistic constraints.
\subsection{General framework}\label{generalframe}
We consider the following general optimization problem with probabilistic constraints:
\begin{equation}\label{optprob}
\min\limits_{u\in U}\,\,F(u)\mbox{ s.t. }\varphi(u)\geq p\quad (p\in (0,1]),
\end{equation}
{ where $F\colon U\to\mathbb{R}$ is some cost function and $\varphi:U\to\mathbb{R}$ denotes a probability function defined by 
\[
\varphi (u):=\mathbb{P}(\omega\mid g(u,\xi(\omega))\leq 0)
\]
with $g\colon U\times\mathbb{R}^m\to\mathbb{R}$ being some constraint function. The following assumption will apply throughout this section:
\begin{assumption}
\label{ass:general}
\subasu $U$ is a reflexive and separable Banach space, \label{subasu:U}
\subasu $F$ is convex and Fréchet differentiable, \label{subasu:F}
\subasu $\xi$ is an $m$-dimensional Gaussian random vector defined on a probability space $(\Omega, \mathcal{F}, \pP)$ and having a centered Gaussian distribution $\mathcal{N}(0,\Sigma)$ with nondegenerate covariance matrix $\Sigma$, \label{subasu:xi}\subasu  $g$ is locally Lipschitzian and $g(u,\cdot )$ is convex for all $u\in U$. \label{subasu:g}
\end{assumption}
}
\noindent
As $\varphi$ is not given analytically, it is important to derive some analytical properties, { specifically, Lipschitz continuity and convexity}, from the given data $(g,\xi)$. To this end, we will fix a point of interest $\bar{u}\in U$ and require two additional assumptions at that point. The first one requires that the mean zero of the distribution of $\xi$ is a Slater point of the convex inequality $g(\bar{u},z)\leq 0$:
\begin{equation}\label{sp}
g(\bar{u},0)< 0.   
\end{equation}
We note that this assumption is not restrictive. Indeed, requiring a Slater point is necessary in order to guarantee that the probability function $\varphi$ is at least continuous. But, if some Slater point exists at all, then the mean zero will be a Slater point too whenever $\varphi(\bar{u})\geq 0.5$, see \cite[Proposition 3.11]{vanAckooij_Henrion_2014}. On the other hand, $\varphi(\bar{u})$ is typically close to one in probabilistic programming.

{ Assumption~\ref{ass:general}} along with \eqref{sp} guarantee that $\varphi$ is (strongly) continuous at $\bar{u}$ (see, e.g., \cite[Theorem 1]{Hantoute2019}). However, the fact that $g$ is even locally Lipschitzian does not yet imply the same property for $\varphi$ despite the nice probability distribution (multivariate Gaussian), see \cite[Example 1]{Hantoute2019}. Similarly, Fréchet differentiability of $g$ does not imply the same property of $\varphi$, see \cite[Proposition 2.2]{vanAckooij_Henrion_2014}. These counterexamples are based on the fact that at the point of interest $\bar{u}\in U$, the set
\begin{equation}\label{bounded-realizations}
M:=\{z\in\mathbb{R}^m\mid g(\bar{u},z)\leq 0\}
\end{equation}
of feasible realizations of the random vector may be unbounded. 
In order to derive the local Lipschitz continuity of $\varphi$ in our setting and then to estimate its Clarke subdifferential, we need the following condition of moderate growth to be satisfied at $\bar{u}$: 
\begin{equation}\label{growth}
\exists l>0\,\,\forall h\in U:g^{\circ}(\cdot ,z)(u;h)\leq l \left\Vert z\right\Vert_2^{-m}\exp\left(\frac{\left\Vert z\right\Vert_2^2}{2\left\Vert \Sigma^{1/2}\right\Vert^2} \right)\left\Vert h\right\Vert_U \ \forall
u\in \mathbb{B}_{1/l}\left(\bar{u} \right), \forall z: \left\Vert z\right\Vert_2\geq l.
\end{equation}
Here, $g^{\circ}(\cdot ,z)(u;h)$ refers to the Clarke directional derivative of the locally Lipschitzian (by { Assumption~\ref{ass:general}}) partial function $g(\cdot,z)$ at the argument $u$ in direction $h$. Moreover, $\Sigma^{1/2}$ denotes a root of $\Sigma$.

Next, we introduce a radial probability function $e\colon U\times\mathbb{S}^{m-1}\to\mathbb{R}$ defined by 
\begin{equation}\label{radprob}
e(u,v):=\mu_\eta\{r\geq 0\mid g(u,r\Sigma^{1/2}v)\leq 0\},
\end{equation}
where { $\mathbb{S}^{m-1}=\{v\in\mathbb{R}^m\mid \|v\|_2=1\}$ is the unit sphere in $\mathbb{R}^m$}, $\mu_\eta$ is the one-dimensional chi distribution with $m$ degrees of freedom. The following representation of the total probability function $\varphi$ as a spherical integral over the radial probability function $e$ is a consequence of the well-known {\em spherical radial decomposition} of Gaussian random vectors: 
\begin{equation}\label{srd}
\varphi (u)=\int\limits_{\mathbb{S}^{m-1}}e(u,v) \D \mu_\zeta (v)\quad (u\in U).
\end{equation}
Here, $\mu_\zeta$ refers to the uniform distribution on $\mathbb{S}^{m-1}$.
The following result on subdifferentiation under the integral sign holds true:
\begin{theorem}[\cite{Hantoute2019}, Theorem 5, Corollary 2 and Proposition 6]\label{probderiv}
In addition to { Assumption~\ref{ass:general}}, let \eqref{sp} and \eqref{growth} be satisfied at some $\bar{u}\in U$ (instead of \eqref{growth}, we may alternatively assume that the set $M$ in \eqref{bounded-realizations} is bounded). Then, the partial radial probability functions $e(\cdot ,v)$, $v\in\mathbb{S}^{m-1}$, are uniformly locally Lipschitzian around $\bar{u}$ with a Lipschitz constant independent of $v$. Moreover, the (total) probability function $\varphi$ is locally Lipschitzian around $\bar{u}$ and its Clarke subdifferential at $\bar{u}$ can be estimated from above by
\begin{equation}\label{clarkeinclu}
\partial^C\varphi (\bar{u})\subseteq\int\limits_{\mathbb{S}^{m-1}}\partial^C_ue(\bar{u},v) \D \mu_\zeta (v),
\end{equation}
where $\partial^C$ is the Clarke subdifferential and $\partial^C_u$ the partial Clarke subdifferential with respect to $u$. If, in addition, the condition
\begin{equation}\label{measurezero2}
\mu_\zeta\{v\in\mathbb{S}^{m-1}\mid\# \partial^C_ue(\bar{u},v)\geq 2\}=0
\end{equation}
is satisfied, then the Clarke subdifferential of $\varphi$ at $\bar{u}$ reduces to a singleton and equality holds in \eqref{clarkeinclu}. As a consequence, $\varphi$ is strictly differentiable at $\bar{u}$ in the Hadamard sense \cite[p.~30]{clarke1983}.
\end{theorem}

\noindent
The integral above is to be understood in the sense of Aumann. In particular, the asserted inclusion means that for each $x^*\in\partial^C\varphi (\bar{u})$ there exists a measurable selection $\hat{x}^*_v\in\partial^C_ue(\bar{u},v)$ (for $\mu_\zeta$-a.e. $v\in\mathbb{S}^{m-1}$) such that 
\[
x^*(h)=\int\limits_{\mathbb{S}^{m-1}}\hat{x}^*_v(h) \D \mu_\zeta (v)\quad\forall h\in U.
\]
At this point one might wonder why the inclusion in Theorem \ref{probderiv} is not sharp apart from the differentiable situation under condition \eqref{measurezero2}. The reason lies in Clarke's theorem on subdifferentiation of integral functionals \cite[Theorem 2.7.2]{clarke1983}, which provides a sharp estimate only in case that the integrand is Clarke regular. It is easily seen that the radial probability function fails to be Clarke regular in general. However, at this point one may pick up a remark by Clarke \cite[Remark 2.3.5]{clarke1983} pointing to the fact that sharp results may be shown also for functions being the negative of a Clarke regular one. This observation allows us in the following, by slightly strengthening our assumptions, to derive an identity rather than upper estimate in the formula of Theorem \ref{probderiv}.
In order to prepare this argument, we observe that by continuity, 
\eqref{sp} holds locally, i.e., there exists some neighborhood $\mathcal{N}(\bar{u})$ of $\bar{u}$ with $g(u,0)<0$ for all $u\in\mathcal{N}(\bar{u})$. Then, 
thanks to the convexity assumption in { Assumption~\ref{ass:general}}, for each $u\in\mathcal{N}(\bar{u})$ and each $v\in\mathbb{S}^{m-1}$, there exists a unique (possibly infinite) {$\rho (u,v)\in\mathbb{R}\cup\{\infty\}$} such that 
\begin{equation}\label{rhodef}
\{r\geq 0\mid g(u,r\Sigma^{1/2}v)\leq 0\}=[0,\rho(u,v)].
\end{equation}
If, in particular, $\rho (u,v)<\infty$, then this value is the unique solution in $r$ of the equation
\begin{equation}\label{uniquesolution}
g(u,r\Sigma^{1/2}v)=0.    
\end{equation}
In other words, we have defined a ``radius function" $\rho\colon \mathcal{N}(\bar{u})\times\mathbb{S}^{m-1}\to\mathbb{R}\cup\{\infty\}$ such that the radial probability function $e$ from \eqref{radprob} may be represented locally around $\bar{u}$ by
\begin{equation}\label{erep}
e(u,v)=\left\{
\begin{array}{ll}
F_\eta(\rho(u,v))&\mbox{if }\rho (u,v)<\infty\\
1&\mbox{if }\rho (u,v)=\infty
\end{array}\right.\quad\forall u\in\mathcal{N}(\bar{u})\,\,\forall v\in\mathbb{S}^{m-1},
\end{equation}
where $F_\eta$ refers to the cumulative distribution function of the (one-dimensional) chi distribution with $m$ degrees of freedom. We are now in a position to establish the Clarke regularity of the negative radial probability function.
\begin{lemma}\label{clarkereg}
In addition to { Assumption~\ref{ass:general}}, suppose that the constraint function $g$ is jointly convex in both variables. Consider some 
$\bar{u}\in U$ with \eqref{sp} and assume that the set $M$ in \eqref{bounded-realizations} is bounded. Then, for each $v\in\mathbb{S}^{m-1}$, the function $-e(\cdot ,v)$ is Clarke regular on the open neighborhood $\mathcal{N}(\bar{u})$ introduced above.
\end{lemma}
\begin{proof}
By our assumptions and by \eqref{erep}, the radial probability function has the representation
\[
e(u,v)=F_\eta(\rho(u,v))\quad\forall u\in\mathcal{N}(\bar{u})\,\,\forall v\in\mathbb{S}^{m-1}.
\]
We observe first that for each fixed $v\in\mathbb{S}^{m-1}$, the partial radius function $\rho (\cdot ,v)\colon \mathcal{N}(\bar{u})\to\mathbb{R}$ is concave. Indeed, let $u^1,u^2\in \mathcal{N}(\bar{u})$ and $\lambda\in [0,1]$ be arbitrary. Then, by \eqref{uniquesolution},
\[
g(u^1,\rho (u^1,v)\Sigma^{1/2}v)=g(u^2,\rho (u^2,v)\Sigma^{1/2}v)=0
\]
and thus, by convexity of $g$,
\[
g(\lambda u^1+(1-\lambda)u^2,\lambda\rho (u^1,v)\Sigma^{1/2}v+(1-\lambda)\rho (u^2,v)\Sigma^{1/2}v)\leq 0.
\]
Now, \eqref{rhodef} yields that
\[
\lambda\rho (u^1,v)+(1-\lambda)\rho (u^2,v)\leq\rho (\lambda u^1+(1-\lambda)u^2,v).
\]
This shows the asserted concavity statement. Next, we extend the cumulative distribution function for the chi distribution from $\mathbb{R}_+$ to $\mathbb{R}$ via
\[
F_\eta^*(t):=\begin{cases}F_\eta(t) &\quad t\geq 0,\\-F_\eta (-t)&\quad t<0.\end{cases}
\]
Since for $t>0$, the derivative of $F_\eta$ equals the  probability density function for the chi distribution, which for arbitrary degrees of freedom is a continuous function, it follows readily by construction that 
$F_\eta^*$ is continuously differentiable. Moreover, $F'_\eta (t)\geq 0$ for all $t\geq 0$ because $F_\eta$ is a distribution function. This implies that $(F_\eta^*)'(t)\geq 0$ for all $t\in\mathbb{R}$. We now have that
\[
-e(u,v)=-F_\eta(\rho(u,v))=F_\eta^*(-\rho(u,v))\quad\forall u\in\mathcal{N}(\bar{u})\,\,\forall v\in\mathbb{S}^{m-1}.
\]
Let us fix an arbitrary $v\in\mathbb{S}^{m-1}$. As we have shown that $-\rho (\cdot ,v)\colon \mathcal{N}(\bar{u})\to\mathbb{R}$ is convex, the same function is Clarke regular \cite[Proposition 2.3.6]{clarke1983}. Therefore, $-e(u,v)$ is the composition of a continuously differentiable, hence Clarke regular (see \cite[Proposition 2.3.6]{clarke1983}) function having non-negative derivative everywhere with another Clarke regular function and is therefore itself Clarke regular \cite[Theorem 2.3.9 (i)]{clarke1983}.
\end{proof}

\begin{theorem}\label{exactformula}
Under the assumptions of Lemma \ref{clarkereg}, one has the exact formula  
\[
\partial^C\varphi (\bar{u})=\int\limits_{\mathbb{S}^{m-1}}\partial^C_ue(\bar{u},v) \D \mu_\zeta (v).
\]  
\end{theorem}
\begin{proof}
Under the assumptions we made, the function $e$ is continuous on $\mathcal{N}(\bar{u})\times\mathbb{S}^{m-1}$ \cite[Proposition 1]{Hantoute2019}. In particular, $-e(u,\cdot )\colon \mathbb{S}^{m-1}\to\mathbb{R}$ is measurable for each $u\in\mathcal{N}(\bar{u})$. Moreover, according to Theorem \ref{probderiv}, the functions $e(\cdot ,v)$ and, hence, $-e(\cdot ,v)$ with $v\in\mathbb{S}^{m-1}$, are uniformly locally Lipschitzian around $\bar{u}$ with a Lipschitz constant independent of $v$. Since constants are integrable with respect to the uniform measure on the sphere, all assumptions of Theorem 2.72 in~\cite{clarke1983} are satisfied for the relation 
\[
-\varphi (u)=\int\limits_{\mathbb{S}^{m-1}}-e(u,v) \D \mu_\zeta (v)\quad (u\in U),
\]
which is equivalent to \eqref{srd}. Accordingly, the mentioned theorem ensures the inclusion
\[
\partial^C(-\varphi ) (\bar{u})\subseteq\int\limits_{\mathbb{S}^{m-1}}\partial^C_u(-e)(\bar{u},v) \D \mu_\zeta (v).
\]
By the Clarke regularity of the functions $-e(\cdot ,v)$ on $\mathcal{N}$ for each $v\in\mathbb{S}^{m-1}$ (shown in Lemma \ref{clarkereg}), the same theorem even guarantees that this inclusion reduces to an equality. Then, however, we may exploit the linearity of Clarke's subdifferential in order to show that
\[
\partial^C\varphi (\bar{u})=-\partial^C(-\varphi) (\bar{u})=-\int\limits_{\mathbb{S}^{m-1}}\partial^C_u(-e)(\bar{u},v) \D \mu_\zeta (v)=\int\limits_{\mathbb{S}^{m-1}}\partial^C_ue(\bar{u},v) \D \mu_\zeta (v).
\]
\end{proof}
Apart from Lipschitz continuity or even differentiability of the constraint, it would be interesting to know what convexity properties are inherent to problem \eqref{optprob}. Given the assumed convexity of the objective, it is tempting to verify the concavity of $\varphi$ in order to arrive at a fully convex problem. Unfortunately, being bounded by zero and one, probability functions are hardly ever concave (they would have to be constant then which is extremely untypical). Nonetheless, some generalized concavity, so-called log-concavity, may hold true, which allows one to reformulate the problem as a convex one. Indeed, one may derive (e.g., \cite[Proposition 4]{Farshbaf-Shaker2018}) from well-known results by Prékopa (see \cite[Theorem 10.2.1]{prekopa1995}), the following statement:

If $g$ is quasi-convex (jointly) in both variables and if the density of $\xi$ is log-concave, then the probability function $\varphi$ is log-concave too, i.e., $\log\varphi$ is concave under the convention $\log 0:=\infty$.
\begin{lemma}\label{convproperties}
If $g$ is quasi-convex (jointly) in both variables, then the optimization problem \eqref{optprob} can be reformulated as the convex optimization problem
\begin{equation}\label{convoptprob}
\min_{u \in U} \{F(u)\mid\tilde{\varphi}(u)\leq 0\}\quad (p\in (0,1]),
\end{equation}
where $\tilde{\varphi}\colon U\to\mathbb{R}\cup \{\infty\}$ is an extended-valued convex function defined by $\tilde{\varphi}(u):=-\log\varphi(u)+\log p$ for $u\in U$ under the convention $-\log 0:=\infty$. Moreover, the function $\tilde{\varphi}$ is lower semicontinuous. 
\end{lemma}
\begin{proof}
Our conventions on the $\log$ along with the assumption that $p>0$ yield the equivalence
\[
\varphi (u)\geq p\Longleftrightarrow -\log\varphi (u)+\log p\leq 0\quad\forall u\in U.
\]
Consequently, \eqref{optprob} and \eqref{convoptprob} are equivalent optimization problems. Since the density of our Gaussian random vector $\xi$ is log-concave, the statement preceding this lemma yields that $\tilde{\varphi}$ is (possibly extended-valued) convex and, hence, \eqref{convoptprob} is a convex optimization problem.
As mentioned above, { Assumption~\ref{ass:general}} along with \eqref{sp} guarantee that 
the probability function $\varphi$ is strongly continuous. As a consequence, for each $t\in\mathbb{R}$, the upper level sets
\[
\{u\in U\mid\varphi (u)\geq t\}
\]
are (strongly) closed in $U$. Now, let $\tau\in\mathbb{R}$ be arbitrarily given. Then, the identity (again thanks to $p>0$)
\[
\{u\in U\mid\tilde{\varphi} (u)\leq\tau\}=\{u\in U\mid\varphi (u)\geq pe^{-\tau}\}
\]
yields the closedness of the left-hand side set thanks to the closedness of the right-hand side upper level set (with respect to the level $t:=pe^{-\tau}$). Since $\tau$ was arbitrary, we have shown that all lower level sets of $\tilde{\varphi}$ are (strongly) closed, which is the same as saying that $\tilde{\varphi}$ is (strongly) lower semicontinuous on $U$.
\end{proof}

\noindent
Now, we are in a position to formulate necessary and sufficient optimality conditions for the originally nonconvex problem \eqref{optprob}. 

\begin{proposition}\label{optcond}
We strengthen { Assumption~\ref{ass:general}} by assuming that the constraint function $g$ is jointly convex in \underline{both} variables. Let there exist some $\hat{u}\in U$ with $\varphi (\hat{u})>p$ (generalized Slater point). Consider some $u^*\in U$ that is feasible in \eqref{optprob}. { For $u^*$, we assume that \eqref{sp} is satisfied and \eqref{growth} is satisfied or the set $M$ in \eqref{bounded-realizations} is bounded.} Then, for $u^*$ to be a solution to \eqref{optprob}, it is \underline{necessary and sufficient} that there exists some $\lambda\geq 0$ with
\begin{equation}\label{kkt}
D F(u^*)\in\lambda\partial^C\varphi (u^*),\quad { \varphi(u^*)\geq p,} \quad\lambda (\varphi(u^*)-p)=0, 
\end{equation}
{ Moreover,} for $u^*$ to be a solution to \eqref{optprob}, it is \underline{necessary} that there exists some $\lambda\geq 0$ such that
\begin{equation}\label{kkt2}
D F(u^*)\in\lambda\int\limits_{\mathbb{S}^{m-1}}\partial^C_ue(u^*,v) \D \mu_\zeta (v),\quad { \varphi(u^*)\geq p,}  \quad\lambda (\varphi(u^*)-p)=0,
\end{equation}
where $e$ refers to the radial probability function introduced in Theorem \ref{probderiv}. { If, moreover, \eqref{measurezero2} is satisfied or the set $M$ in \eqref{bounded-realizations} is bounded,} then \eqref{kkt2} is also \underline{sufficient} for $u^*$ to be a solution to \eqref{optprob}. 
\end{proposition}
\begin{proof}
The point $u^*$ being a solution to \eqref{optprob} is equivalent to it being a solution to the convex problem \eqref{convoptprob}. Due to $\tilde{\varphi}(\hat{u})=-\log\varphi(\hat{u})+\log p<0$, $\hat{u}$ is a Slater point for \eqref{convoptprob}. 
Hence, by Lemma \ref{convproperties}, $\tilde{\varphi}$ is a lower semicontinuous, convex and proper function. Now,
it follows from classical convex analysis (admitting extended-valued inequality constraints, e.g., \cite[Theorem 3.66, Remark 3.67]{peypouquet2015convex}) that $u^*$ being a solution to \eqref{convoptprob} (and \eqref{optprob}) is equivalent to the existence of some $\lambda\geq 0$ such that $-D F(u^*)\in\lambda\partial\tilde{\varphi} (u^*)$ and $\lambda\tilde{\varphi}(u^*)=0$,
where ``$\partial$'' refers to the subdifferential in the sense of convex analysis. Now, the feasibility of $u^*$ implies that $\varphi(u^*)\geq p>0$. With $\varphi$ being locally Lipschitzian around $u^*$ (by Theorem \ref{probderiv})) and $\log$ being locally Lipschitzian around the positive number $\varphi(u^*)$, the function  $\tilde{\varphi}$ is locally Lipschitzian around $u^*$.
Since $\tilde{\varphi}$ is also convex, it follows that
$\partial\tilde\varphi(u^*)=\partial^C\tilde\varphi(u^*)$, see \cite[Proposition 2.2.7]{clarke1983}. On the other hand, the chain rule for Clarke's subdifferential from  \cite[Theorem 2.3.9 (ii)]{clarke1983}
yields that 
\[
\partial\tilde\varphi(u^*)=\partial^C\tilde\varphi(u^*)=-\frac{1}{\varphi(u^*)}\partial^C\varphi(u^*).
\]
Owing to $\varphi(u^*)>0$, we have shown that the existence of some $\lambda\geq 0$ with $-D F(u^*)\in\lambda\partial\tilde{\varphi} (u^*)$ and $\lambda\tilde{\varphi}(u^*)=0$ is equivalent with the existence of some $\tilde{\lambda}\geq 0$ with $D F(u^*)\in\tilde{\lambda}\partial^C\varphi (u^*)$ and $\tilde{\lambda}(\varphi(u^*)-p)=0$. Altogether, this yields the first statement of this Proposition. The remaining statements on necessity and sufficiency of \eqref{kkt2} are immediate consequences of \eqref{kkt} and of Theorems \ref{probderiv} and \ref{exactformula}, respectively.
\end{proof}
\begin{remark}
\label{rem:case-p-equals-1}
The existence of some $\hat{u}$ with $\varphi(\hat{u})>p$ as required in Proposition \ref{optcond} implicitly necessitates that $p<1$, which means that the ``almost sure'' case is excluded. Without this requirement, the lack of a Slater point in the convex problem \eqref{convoptprob} would not allow the derivation of optimality conditions and simple examples show that they do not apply then, indeed. The case $p=1$ will be addressed in Section \ref{sec:almost-sure} in an independent manner.
\end{remark}
{
\begin{remark}\label{without-joint-convexity}
The assumption of joint convexity of $g$ in Proposition 
\ref{optcond} appears to be restrictive compared with our general assumption of convexity just in the second argument (Assumption \ref{ass:general}) as it rules out, for instance, bilinear couplings of control and randomness. On the other hand, joint convexity holds true, if $g$ happens to be linear as in our concrete PDE-constrained
optimization problem considered in the following sections. The weaker assumption of convexity with respect to the second argument, only, would allow to carry out the same analysis except that the exact formula for the Clarke subdifferential from Theorem \ref{exactformula} would get lost and, as a consequence, the sufficiency of optimality condition \eqref{kkt2} would hold true only under the special differentiability condition \eqref{measurezero2}.
\end{remark}

\noindent
The formulation of necessary and possibly even sufficient optimality conditions in terms of the radial probability function as in \eqref{kkt2} paves the way for deriving fully explicit optimality conditions in terms of the original problem data; this is possible to do if the partial subdifferential $\partial^C_ue$ can be computed. We will demonstrate how to do this in the context of a PDE-constrained optimization problem subject to uniform random state constraints in the following subsections. The application of \eqref{kkt2}, however, is not restricted to this PDE setting and could be of equal use in many other operations research problems. One example is reservoir management under uncertain inflow and controlled release subject to level constraints uniformly in time. These reservoirs might refer to hydro or gas reservoirs in energy management, to battery storage in the dispatch of mini-grids, or to pension funds in finance. Reservoir problems subject to probabilistic constraints have been considered in many papers, mostly from the numerical context without reference to optimality conditions and with finite rather than infinite random inequality systems involved. In \cite{Berthold2022}, however, a water reservoir problem has been investigated with level constraints uniformly on a time interval (corresponding to uniform state constraints on a domain in our PDE context). This setting would perfectly fit to our general framework studied above. Then, computing $\partial^C_ue$ as we will now do in the PDE-constrained context would lead to analogous optimality conditions.}
\subsection{Probabilistic state constraints in the concrete PDE}
\label{subsec:prob-PDE}
In this section, we are going to apply the general results obtained before to the specific optimization problem
\begin{subequations}
    \begin{alignat}{3}
    \min_{u \in L^2(D)}\, F(u) & & &&   \label{eq:probuniform-problem-a}\\
    \text{s.t.}  \quad -\Delta  y(x,\omega)  &=  u(x) + f(x,\xi (\omega)), & &&\quad x \in D \quad \text{ $\mathbb{P}$-a.s.},  \label{eq:probuniform-problem-b} \\
 y(x,\omega) &=0, & &&\quad x \in \partial D  \quad \text{$\mathbb{P}$-a.s.},  \label{eq:probuniform-problem-c}\\
 \mathbb{P}(y(x,\omega)&\leq \alpha\quad\forall x\in D)\geq p. & &&  \label{eq:probuniform-problem-d}
    \end{alignat}
    \label{eq:probuniform-problem}
\end{subequations}

\noindent
We recall the setting introduced in Section~\ref{subsec:PDE} under Assumption~\ref{ass:PDE-standing}. This setting provides the continuity of the solution operator and the measurability of the functions $\omega\mapsto y(x,\omega )$ for each $x\in\bar{D}$ via Lemma \ref{measstat}. This justifies relations \eqref{eq:probuniform-problem-b} and \eqref{eq:probuniform-problem-c}. As for \eqref{eq:probuniform-problem-d}, a corresponding measurability statement will be made in Lemma \ref{baseprop} below.
Here, $U = L^2(D)$, { meaning Assumption~\ref{subasu:U} is satisfied}; { $F$ and $\xi$ fulfill Assumptions~\ref{subasu:F}--\ref{subasu:xi}}, %$\xi\sim\mathcal{N}(0,\Sigma)$ is a centered $m$-dimensional Gaussian random vector defined on the probability space $(\Omega,\mathcal{F},\mathbb{P})$,
$p\in (0,1]$ is some given probability level, and $\alpha\in\mathbb{R}$ is some upper threshold for the random state $y(\cdot,\omega)$. Note that $\mathbb{E}f(\cdot,\xi )=f_0$ on account of $\xi$ being centered (see \eqref{eq:f}). 

Recall the operators defined in \eqref{affinedecomp}, allowing us to reformulate problem \eqref{eq:probuniform-problem} as
\begin{equation}\label{reduced_problem}
\min_{u \in L^2(D)}\, F(u)\quad\mbox{s.t.}\quad\varphi (u)\geq p,
\end{equation}
where
\[
\varphi (u):=\mathbb{P}(\omega\mid g(u,\xi(\omega))\leq 0)\quad (u\in L^2(D))\quad\mbox{and}\quad g(u,z):=\sup\limits_{x\in D}\,\, [S(u,z)](x)-\alpha\quad ((u,z)\in L^2(D)\times\mathbb{R}^m).
\]
Problem \eqref{reduced_problem} falls into the setting of the general problem \eqref{optprob} introduced before. 
In order to apply the results from Section \ref{generalframe} to problem \eqref{reduced_problem}, one has to verify first that the assumptions made there are satisfied. 
Since by Lemma~\ref{lem:cont-dep-sol}, $S(u,z)\in H^1_0(D)\cap\mathcal{C}(\bar{D})$ for all $(u,z)\in L^2(D)\times\mathbb{R}^m$, 
%\RH{Ist es nicht eher die Stetigkeit von $S$?}, 
%. In particular,  $S(u,z)\in H^1_0(D)\cap\mathcal{C}(\bar{D})$ for all $(u,z)\in L^2(D)\times\mathbb{R}^m$. As a consequence, 
we may write $g$ as a maximum over $\bar{D}$ rather than a supremum over $D$:
\[
g(u,z)=\max\limits_{x\in\bar{D}}\,\, [S(u,z)](x)-\alpha\leq 0\quad ((u,z)\in L^2(D)\times\mathbb{R}^m).
\]
Next, we verify that $g$ satisfies { Assumption~\ref{subasu:g}} of Section \ref{generalframe}:
\begin{lemma}\label{baseprop}
The function $g\colon L^2(D)\times\mathbb{R}^m\to\mathbb{R}$ is convex and globally Lipschitzian. As a consequence, the coinciding events occurring in the probabilities of \eqref{eq:probuniform-problem-d} and in the definition of $\varphi$ above are measurable.
\end{lemma}
\begin{proof}
The convexity property follows directly from the fact that $g$ is the maximum of affine linear functions 
\[(u,z)\mapsto [S(u,z)](x)=[P(u,z)](x)+y_0(x)-\alpha\quad  (x\in\bar{D}). 
\]
As for Lipschitz continuity, let $(u_1,z_1),(u_2,z_2)\in L^2(D)\times\mathbb{R}^m$ be arbitrary. Then,
\[
|g(u_1,z_1)-g(u_2,z_2)|=\Big\vert\max\limits_{x\in\bar{D}}\,[S(u_1,z_1)](x)-\max\limits_{x\in\bar{D}}\,[S(u_2,z_2)](x)\Big\vert=\vert [S(u_1,z_1)](x_1^*)-[S(u_2,z_2)](x_2^*)\vert,
\]
where $x_1^*,x_2^*\in\bar{D}$ are arguments realizing the respective maxima. Assuming w.l.o.g.~$S(u_1,z_1)](x_1^*)\geq S(u_2,z_2)](x_2^*)$, we may exploit the linearity of $P$ and refer to \eqref{pest} in order to derive that
\begin{eqnarray*}
|g(u_1,z_1)-g(u_2,z_2)|&\leq& [S(u_1,z_1)](x_1^*)-[S(u_2,z_2)](x_1^*)=[P(u_1,z_1)](x_1^*)-[P(u_2,z_2)](x_1^*)\\&=&[P(u_1-u_2,z_1-z_2)](x_1^*)\leq\norm{P(u_1-u_2,z_1-z_2)}{\mathcal{C}(\bar{D})}\\&\leq& C(\|u_1-u_2\|_{L^2(D)}+\|z_1-z_2\|_2).
\end{eqnarray*}
This proves the (global) Lipschitz continuity of $g$ and in particular the remaining measurability statement.
\end{proof}

\noindent
The application of Theorem \ref{probderiv} requires, in particular, that the growth condition \eqref{growth} is satisfied or, alternatively, that the set $M$ in \eqref{bounded-realizations} is bounded. The latter condition may not hold in general as can be seen from Examples \ref{illustex} and \ref{illustex2} below, where 
the radius function $\rho$ from \eqref{rhodef} becomes infinite for certain directions. This observation urges us to verify the general validity of the growth condition \eqref{growth} before applying Theorem \ref{probderiv}.
\begin{lemma}\label{growthcheck}
Let $\bar{u}\in L^2(D)$ be arbitrary. Then, $g$ satisfies the condition of moderate growth \eqref{growth}.
\end{lemma}
\begin{proof}
By definition of Clarke's directional derivative, we have that
\[
g^{\circ}(\cdot ,z)(u;h)=\limsup\limits_{t\downarrow 0,u'\to u}\frac{g(u'+th,z)-g(u',z)}{t}\quad\forall u,h\in L^2(D)\,\forall z\in\mathbb{R}^m.
\]
Exploiting Lemma \ref{baseprop} with the Lipschitz constant $C$ derived there, we may continue as
\begin{equation}\label{dirderiv}
g^{\circ}(\cdot ,z)(u;h)\leq C\norm{h}{L^2(D)}\quad\forall u,h\in L^2(D)\,\forall z\in\mathbb{R}^m.
\end{equation}
Define $L>0$ such that
\[
\left\Vert z\right\Vert_2^{-m}\exp\left(\frac{\left\Vert z\right\Vert_2^2}{2\left\Vert \Sigma^{1/2}\right\Vert^2}\right)\geq C\quad\forall z\in\mathbb{R}^m:\|z\|_2\geq L
 \]
(note that the left-hand side expression tends to infinity if $\|z\|$ does so). Then, \eqref{growth} is satisfied with $l:=\max\{1,L\}$.
\end{proof}
\begin{corollary}\label{cor-theo}
Fix some point of interest $\bar{u}\in L^2(D)$. Denote by $\bar{y}:=S(\bar{u},0)$ the unique solution of the PDE \eqref{eq:PDE} associated with $\bar{u}$ and $z:=0$ (expectation of $\xi$). If for this expected state, it holds that
\begin{equation}\label{sp1}
\bar{y}(x)<\alpha\quad\forall x\in\bar{D},
\end{equation}
then the conclusions of Theorem \ref{probderiv} hold true.
\end{corollary}
\begin{proof}
The assumption implies that 
\[
g(\bar{u},0)=\max\limits_{x\in\bar{D}}\,\, [S(\bar{u},0)](x)-\alpha
=\max\limits_{x\in\bar{D}}\,\,\bar{y}(x)-\alpha <0.
\]
This, however, is \eqref{sp}. Since all the remaining assumptions of Theorem \ref{probderiv} are satisfied by Lemmas \ref{baseprop} and \ref{growthcheck}, the result follows.
\end{proof}
\begin{corollary}\label{coroptcond}
In the setting of Corollary \ref{cor-theo} assume in addition to \eqref{sp1} (a Slater condition in $\mathbb{R}^m$) the existence of some  $\hat{u}\in U$ with $\varphi (\hat{u})>p$ (generalized Slater point in $U$). Then, the results of Proposition \ref{optcond} hold true in the concrete optimization problem \eqref{reduced_problem}.
\end{corollary}
\begin{proof}
Since the function $g$ associated with the probability function $\varphi$ in \eqref{reduced_problem} is convex (in both variables simultaneously) by Lemma \ref{baseprop} and all other needed assumptions are satisfied as already stated in the proof of Corollary \ref{cor-theo}, the assertion follows from Proposition \ref{optcond}. 
\end{proof}

In Section \ref{subsec:subdiff-prob-func-PDE}, we will provide a more explicit representation of the optimality conditions for problem \eqref{reduced_problem}. First, we will work on formulas of the radial probability function.

\subsection{Explicit formulae for the subdifferential of the radial probability function}
The general results of Theorem \ref{probderiv} and Proposition \ref{optcond} on the subdifferential of the (total) probability function $\varphi$ and on optimality conditions for problem \eqref{optprob} are formulated in terms of the radial probability function $e$ and in this sense not very explicit yet. In this section we provide two lemmas that characterize the (partial) Clarke subdifferential of $e$ in terms of the original data of our concrete problem \eqref{eq:probuniform-problem} (or \eqref{reduced_problem}). In both lemmas we will impose the Slater condition \eqref{sp1} at the point $\bar{u}$, which implies the Slater condition \eqref{sp} in the general setting. Accordingly, we will consider the neighborhood $\mathcal{N}(\bar{u})$ and the radius function $\rho\colon \mathcal{N}(\bar{u})\times\mathbb{S}^{m-1}\to\mathbb{R}\cup\{\infty\}$ introduced at the end of section \ref{generalframe}.
We have to distinguish two cases, namely, whether at the fixed point $\bar{u}$ the value $\rho (\bar{u},v)$ is finite or infinite for a given direction $v\in\mathbb{S}^{m-1}$. In the latter case, the associated ray in \eqref{rhodef} is unbounded, hence the radial probability $e$ realizes its maximum, which is one. Consequently, the (sub-) derivative at this ray reduces to zero. This is made more precise in the following lemma. 

For the remainder of Section~\ref{sec:prob-constraints}, we will  identify the dual of $U=L^2(D)$ with itself and work with the inner product defined by $(\cdot,\cdot)_{L^2(D)}=\langle \cdot,\cdot\rangle_{L^2(D)^*,L^2(D)}.$ 

\begin{lemma}\label{techlemma1}
As in Corollary \ref{cor-theo}, fix some $\bar{u}\in L^2(D)$ with associated expected state $\bar{y}$ satisfying \eqref{sp1}. If $v\in\mathbb{S}^{m-1}$ is such that $\rho (\bar{u},v)=\infty$, then $\partial^C_ue(\bar{u},v)=\{0\}$ for $e$ defined in \eqref{radprob}.
\end{lemma}
\begin{proof}
Fix $v$ as indicated. As observed in the proof of Corollary \ref{cor-theo}, relation \eqref{sp1} yields \eqref{sp}. Accordingly, we may find a neighborhood $\mathcal{N}(\bar{u})$ of $\bar{u}$ with $g(u,0)\leq \frac{1}{2}g(\bar{u},0)<0$ for all $u\in\mathcal{N}(\bar{u})$. Then, by \cite[Corollary 1 (ii)]{Hantoute2019}, one has that
\begin{equation}\label{infinitycase}
\partial^F_ue(u,v')\subseteq\{0\}\quad\forall u\in\mathcal{N}(\bar{u})\,\,\forall v'\in\mathbb{S}^{m-1}:\rho (u,v')=\infty ,
\end{equation}
where $\partial^F_u$ denotes the partial Fréchet subdifferential with respect to $u$. Now, let $u\in\mathcal{N}(\bar{u})$, $h\in L^2(D)$, $v'\in\mathbb{S}^{m-1}$ with $\rho (u,v')<\infty$ and $y^*\in\partial^F_ue(u,v')$ be arbitrarily given. By \cite[Theorem 2]{Hantoute2019}, there exists 
\[
(u^*,z^*)\in\partial^C_ug(u,\rho (u,v')\Sigma^{1/2}v')\times  \partial^C_zg(u,\rho (u,v')\Sigma^{1/2}v')
\]
such that
\[
( y^*,h)_{L^2(D)}\leq\frac{\chi (\rho (u,v'))}{\langle z^*,\Sigma^{1/2}v'\rangle }(u^*,-h)_{L^2(D)},
\]
where $\chi (t):=Kt^{m-1}e^{-t^2/2}$ $(t\geq 0)$ is the density of the one-dimensional chi distribution with $m$ degrees of freedom and $K$ is some normalizing factor. Moreover, by \cite[Lemma 3]{Hantoute2019},
\[
\langle z^*,\Sigma^{1/2}v'\rangle\geq\frac{-g(u,0)}{\rho (u,v')}=\frac{|g(u,0)|}{\rho (u,v')}>0.
\]
By virtue of the definition of Clarke's generalized (partial) directional derivative and by \eqref{dirderiv}, we may continue with
\begin{eqnarray*}
( y^*,h)_{L^2(D)}&\leq&\frac{\chi (\rho (u,v'))}{\langle z^*,\Sigma^{1/2}v'\rangle }g^0(\cdot ,\rho (u,v')\Sigma^{1/2}v')(u;-h)\leq\frac{\chi (\rho (u,v'))}{\langle z^*,\Sigma^{1/2}v'\rangle }C\norm{h}{L^2(D)}\leq\frac{\chi (\rho (u,v'))\rho (u,v')}{|g(u,0)|}C\norm{h}{L^2(D)}\\
&\leq&\frac{2\chi (\rho (u,v'))\rho (u,v')}{|g(\bar{u},0)|}C\norm{h}{L^2(D)}.
\end{eqnarray*}
Since $h\in L^2(D)$ was arbitrary, it follows that 
\begin{equation}\label{normest}
\norm{y^*}{L^2(D)}\leq\frac{2\chi (\rho (u,v'))\rho (u,v')C}{|g(\bar{u},0)|}\quad\forall u\in\mathcal{N}(\bar{u})\,\,\forall v'\in\{\mathbb{S}^{m-1}:\rho (u,v')<\infty\}\,\,\forall y^*\in\partial^F_ue(u,v').
\end{equation}
In order to prove our claim, it will be sufficient to show that $\partial^M_ue(\bar{u},v)=\{0\}$, where $\partial^M_u$ refers to the (partial) Mordukhovich subdifferential. Indeed, recalling that $e(\cdot,v)$ is locally Lipschitzian around $\bar{u}$ by Theorem \ref{probderiv} and referring to a well-known relation between the Mordukhovich and Clarke subdifferentials \cite[Theorem 3.57]{mordukhovich2006}, we infer that
\[
\partial^C_ue(\bar{u},v)={\rm clco}\{\partial^M_ue(\bar{u},v)\}=\{0\},
\]
where ``clco" denotes the weak closure in $L^2(D)$. Since the Mordukhovich subdifferential is nonempty for locally Lipschitzian functions on Hilbert spaces (in general: Asplund spaces, see \cite[Corollary 1.81]{mordukhovich2006}), it will be sufficient to prove the inclusion $\partial^M_ue(\bar{u},v)\subseteq\{0\}$. In order to do so, choose an arbitrary $\bar{y}^*\in \partial^M_ue(\bar{u},v)$. We have to show that $\bar{y}^*=0$. By definition of the Mordukhovich subdifferential, there exist a strongly convergent sequence $u_k\to\bar{u}$ and a weakly (weak-star in general Asplund spaces) convergent sequence $y_k^*\rightharpoonup\bar{y}^*$ with $y_k^*\in\partial^F_u e(u_k,v)$. Now, define the subsequence $u_{k_l}$ satisfying the condition $\rho (u_{k_l},v)<\infty$. Then, it follows from  \cite[Lemma 1]{Hantoute2019} that, $\rho (u_{k_l},v)\to\infty$. On the other hand, $\lim\limits_{t\to\infty}t\cdot\chi (t)=0$. Along with \eqref{normest}, this yields that $\norm{y^*_{k_l}}{L^2(D)}\to 0$ for this subsequence. The remaining elements of the original sequence, however, satisfy $\rho (u_{k},v)=\infty$ which implies by \eqref{infinitycase} that $y^*_k=0$ for those $k$. Altogether we have shown that the original sequence satisfies $\norm{y^*_{k}}{L^2(D)}\to 0$. Then, the weak sequential lower semicontinuity of the norm provides the desired result $\bar{y}^*=0$. 
\end{proof}

\noindent
The following lemma addresses the alternative case, where the radius function takes a finite value and---contrary to the previous situation---substantial information on the subdifferential of the radial and, thus, on the total probability function is provided.
\begin{lemma}\label{techlemma2}
As in Corollary \ref{cor-theo}, fix some $\bar{u}\in L^2(D)$ with associated expected state $\bar{y}$ satisfying \eqref{sp1}.  If $v\in\mathbb{S}^{m-1}$ is such that $\rho (\bar{u},v)<\infty$, then 
\begin{eqnarray*}
\partial^C_ue(\bar{u},v)&=&-\chi (\rho (\bar{u},v))\cdot{\rm clco}\,\left\{\left.
\frac{1}{[P(0,\Sigma^{1/2}v)](x)}\cdot u_x\right| x\in M^*(\bar{u},v)\right\}\\&&\\
M^*(\bar{u},v)&:=&\{x\in\bar{D}\mid [P(0,\Sigma^{1/2}v)](x)>0,\,\,\bar{y}(x)+\rho (\bar{u},v)[P(0,\Sigma^{1/2}v)](x)=\alpha \},
\end{eqnarray*}
where for $x\in\bar{D}$, we denote by 
$u_x\in L^2(D)$ the continuous linear function defined by 
\[
( u_x,h)_{L^2(D)}:= [P(h,0)](x)=(A^{-1}h)(x)\quad (h\in L^2(D)).
\]
\end{lemma}
\begin{proof}
Fix $v$ as indicated and let $\mathcal{N}(\bar{u})$ be the neighborhood of $\bar{u}$ as defined above \eqref{rhodef}. With the concrete form of $g$ in our PDE-constrained problem \eqref{eq:probuniform-problem} and the affine representation \eqref{affinedecomp}, for each $u\in\mathcal{N}(\bar{u})$, $\rho (u,v)$ is the unique (possibly infinite) value such \eqref{rhodef} holds true, which amounts to 
\begin{equation}\label{firstrep}
\left\{r\geq 0 \, \left| \, \max\limits_{x\in\bar{D}}\,\,[P(u,r\Sigma^{1/2}v)](x)+y_0(x)\leq\alpha\right.\right\}=[0,\rho (u,v)]\quad\forall u\in\mathcal{N}(\bar{u}).
\end{equation}
By definition of $\mathcal{N}(\bar{u})$ it holds that $g(u,0)<0$ for all $u\in\mathcal{N}(\bar{u})$ or, equivalently, that
\begin{equation}\label{sp3}
\max\limits_{x\in\bar{D}}\,\,[P(u,0)](x)+y_0(x)<\alpha \quad\forall u\in\mathcal{N}(\bar{u}). 
\end{equation}
Taking into account the linearity of $P$, we may rewrite \eqref{firstrep} as
\begin{equation}\label{secondrep}
\left\{r\geq 0\, \left| \, r\cdot [P(0,\Sigma^{1/2}v)](x)\leq\alpha -y_0(x)-[P(u,0)](x)\quad\forall x\in\bar{D}\right.\right\}=[0,\rho (u,v)]\quad\forall u\in\mathcal{N}(\bar{u}).
\end{equation}
As a consequence of \eqref{sp3}, we have that
\begin{equation}\label{sp4}
\alpha -y_0(x)-[P(u,0)](x)>0\quad\forall x\in\bar{D}\,\,\forall u\in\mathcal{N}(\bar{u}).
\end{equation}
Introduce an open subset of $\bar{D}$ by
\[
\mathcal{O}:=\{x\in\bar{D}\mid [P(0,\Sigma^{1/2}v)](x)>0\}.
\]
Note that $\mathcal{O}\neq\emptyset$ because otherwise the left-hand side of \eqref{secondrep} would be equal to $\mathbb{R}_+$ thanks to \eqref{sp4}. Then, on the right-hand side of \eqref{secondrep} we would end up at the contradiction $\rho (u,v)=\infty$ for all $u\in\mathcal{N}(\bar{u})$ with our assumption $\rho (\bar{u},v)<\infty$.
Now, again by virtue of \eqref{sp4}, equation \eqref{secondrep} can be reformulated as
\begin{equation}\label{thirdrep}
\left\{r\geq 0\, \left| \, r\cdot [P(0,\Sigma^{1/2}v)](x)\leq\alpha -y_0(x)-[P(u,0)](x)\quad\forall x\in\mathcal{O}\right.\right\}=[0,\rho (u,v)]\quad\forall u\in\mathcal{N}(\bar{u}).    
\end{equation}
This implies that
\begin{equation}\label{infrep}
\rho (u,v)=\inf\limits_{x\in\mathcal{O}}\gamma (u,x)\quad\forall u\in\mathcal{N}(\bar{u}),
\end{equation}
where $\gamma\colon L^2(D)\times\mathcal{O}\to\mathbb{R}$ is a continuous function defined by
\[
\gamma (u,x):=\frac{\alpha -y_0(x)-[P(u,0)](x)}{[P(0,\Sigma^{1/2}v)](x)}\quad\forall x\in\mathcal{O}\,\,\forall u\in L^2(D).
\]
In order to verify the continuity of $\gamma$, let $(u_n,x_n)\to (u,x)$ in $L^2(D)\times\mathcal{O}$. 
Then, $|\gamma (u,x)-\gamma (u,x_n)|\to 0$ by definition of the operators $S$ and $P$ and by the fact that $S(u',z')\in\mathcal{C}(\bar{D})$ for all $(u',z')$ (see \eqref{affinedecomp}).
It remains to verify that $|\gamma (u,x_n)-\gamma (u_n,x_n)|\to 0$, which amounts to showing that
\[
\left|\frac{[P(u-u_n,0)](x_n)}{[P(0,\Sigma^{1/2}v)](x_n)}\right|\to 0.
\]
Since $[P(0,\Sigma^{1/2}v)](x_n)\to [P(0,\Sigma^{1/2}v)](x)>0$ by definition of $\mathcal{O}$, it suffices to show that $[P(u-u_n,0)](x_n)\to 0$. This, however, follows readily from \eqref{pest}.

Our aim is to represent $\rho$ not just as an infimum over an open set but as a minimum over a compact set. We introduce the multifunction
\begin{equation}\label{mudef}
M(u):=\{x\in\mathcal{O}\mid\rho (u,v)=\gamma (u,x)\}\quad (u\in\mathcal{N}(\bar{u})).
\end{equation}
We show first that $M$ has nonempty images. For this purpose,  fix an arbitrary $u\in\mathcal{N}(\bar{u})$ and let 
$\{x_n\}\subseteq\mathcal{O}$ be a sequence with $\gamma (u,x_n)\to\rho (u,v)$. Without loss of generality, $x_n\to\bar{x}\in\bar{D}$. Assume that $\bar{x}\notin\mathcal{O}$. Then, $[P(0,\Sigma^{1/2}v)](\bar{x})=0$ because of $x_n\in\mathcal{O}$. At the same time,
\[
\alpha -y_0(\bar{x})-[P(u,0)](\bar{x})>0
\]
by virtue of \eqref{sp4}. It follows that $\gamma (u,x_n)\to\infty$, whence the contradiction $\rho (u,v)=\infty$ with the already mentioned fact that $\mathcal{O}\neq\emptyset$. Therefore, $\bar{x}\in\mathcal{O}$. Now, the continuity of $\gamma$ on $L^2(D)\times\mathcal{O}$ yields that
\[
\gamma (u,x_n)\to\gamma (u,\bar{x})=\rho (u,v).
\]
Hence, $\bar{x}\in M(u)\neq\emptyset$ as was to be shown.

Next, we see that $M(\bar{u})$ is closed, hence compact as a subset of $\bar{D}$. Indeed, let $x_n\in M(\bar{u})$ be a sequence with $x_n\to\bar{x}$. Since $x_n\in\mathcal{O}$, the same reasoning as before allows us to verify that $\bar{x}\in\mathcal{O}$. Then, again by continuity of $\gamma$,
\[
\rho (\bar{u},v)=\gamma (\bar{u},x_n)\to\gamma (\bar{u},\bar{x}),
\]
providing that $\rho (\bar{u},v)=\gamma (\bar{u},\bar{x})$ and, thus, $\bar{x}\in M(\bar{u})$.

Knowing that $M(\bar{u})$ is a nonempty and compact subset of $\mathcal{O}$, there must exist another compact subset $\mathcal{K}\subseteq\mathcal{O}$ such that $M(\bar{u})\subseteq {\rm int}\,\mathcal{K}$. We claim that 
\begin{equation}\label{rhomin}
\rho (u,v)=\min\limits_{x\in\mathcal{K}}\,\,\gamma (u,x)\quad\forall u\in\mathcal{N}(\bar{u}),
\end{equation}
which would be our desired representation of $\rho$ as a minimum over a compact set. For \eqref{rhomin} to hold true it would be sufficient to verify that 
\begin{equation}\label{uniformbound}
M(u)\subseteq\mathcal{K}\quad\forall u\in\mathcal{N}(\bar{u}).
\end{equation}
Indeed, assuming this holds true, we may fix an arbitrary $u\in\mathcal{N}(\bar{u})$ and select some $\hat{x}\in M(u)$ which is possible since this latter set is nonempty, as we have shown before. Then, by \eqref{uniformbound}, $\hat{x}\in\mathcal{K}\subseteq\mathcal{O}$, and so, by definition of $M(u)$,
\[
\rho (u,v)\leq\min\limits_{x\in\mathcal{K}}\,\,\gamma (u,x)\leq\gamma (u,\hat{x})=\rho (u,v),
\]
where the first inequality relies on \eqref{infrep}. This is \eqref{rhomin} and it remains to verify \eqref{uniformbound}. If this relation was false, then there would exist sequences $u_n\to\bar{u}$ and $x_n\in M(u_n)\setminus\mathcal{K}$. Without loss of generality, $x_n\to\bar{x}\in\bar{D}$. Then $\bar{x}\notin {\rm int}\,\mathcal{K}\subseteq\mathcal{O}$ and, by definition of the set $\mathcal{K}$, $\bar{x}\notin M(\bar{u})$. Select some $\hat{x}\in M(\bar{u})\subseteq\mathcal{O}$ (recalling that $M(\bar{u})\neq\emptyset)$. Then,
\[
\gamma (\bar{u},\bar{x})=\lim\limits_n \gamma (u_n,x_n)=\lim\limits_n \rho (u_n,v)\leq\lim\limits_n \gamma (u_n,\hat{x})=\gamma (\bar{u},\hat{x})=\rho (\bar{u},v)<\gamma (\bar{u},\bar{x}),
\]
where the strict inequality comes as a consequence of $\bar{x}\in\mathcal{O}\setminus M(\bar{u})$. This contradiction establishes \eqref{uniformbound}.

Now, defining a function $\kappa :L^2(D)\to\mathbb{R}$ by
\[
\kappa (u):=\max\limits_{x\in\mathcal{K}}\,\,-\gamma (u,x)\quad u\in L^2(D),
\]
we derive from \eqref{rhomin} that 
\begin{equation}\label{kapparho}
 \rho (u,v)=-\kappa (u)\quad\forall u\in\mathcal{N}(\bar{u}).  
\end{equation}
Since $\kappa$ is the maximum of continuous, affine linear functions $-\gamma (\cdot ,x)$ over a compact set, it is convex and finite-valued. Therefore, and taking into account the continuity of $\gamma$, the Ioffe--Tikhomirov Theorem 
\cite[Section 4.2, Theorem 3]{Ioffe-Tikhomirov} allows us to represent the  convex subdifferential of $\kappa$ at $\bar{u}$ as
\[
\partial \kappa (\bar{u})={\rm clco}\,\left\{
\frac{\partial (-\gamma)}{\partial u}(\bar{u},x) \, \Big\vert \, x\in\tilde{M}(\bar{u},v)\right\},\quad\tilde{M}(\bar{u},v):=\{x\in \mathcal{K}\mid -\gamma (\bar{u},x)=\kappa (\bar{u})\}.
\]
With $u_x\in L^2(D)$ as introduced in the statement of this lemma, we arrive at
\[
\frac{\partial (-\gamma)}{\partial u}(\bar{u},x)=\frac{1}{[P(0,\Sigma^{1/2}v)](x)}\cdot u_x\quad\forall x\in\mathcal{O}.
\]

The continuity of $\gamma$ along with the compactness of $\mathcal{K}$ imply via \eqref{rhomin} that $\kappa$ is continuous. Hence, being also convex, $\kappa$ is even locally Lipschitz continuous (cf.~\cite[Prop. 2.107]{Bonnans2013}) and so, its subdifferential in the sense of convex analysis coincides with its Clarke subdifferential \cite[Proposition 2.2.7]{clarke1983}. Applying now a well-known calculus rule for the latter \cite[Proposition 2.3.1]{clarke1983} and taking into account \eqref{kapparho}, we derive that
\begin{equation}\label{almostdone}
\partial_u^C \rho (\bar{u},v)=\partial^C (-\kappa (\bar{u}))=-\partial^C \kappa (\bar{u})=-\partial \kappa (\bar{u})=-{\rm clco}\,\left\{
\frac{1}{[P(0,\Sigma^{1/2}v)](x)}\cdot u_x \, \Big\vert \, x\in\tilde{M}(\bar{u},v)\right\}.
\end{equation}
Since $\rho (\cdot , v)$ is finite-valued on $\mathcal{N}(\bar{u})$ (see \eqref{rhomin}), the representation \eqref{erep} of the radial probability function yields that
$e(u,v)=F_\eta(\rho(u,v))$ locally around $\bar{u}$. With the one-dimensional function $F_\eta$ being continuously differentiable and having derivative $\chi$ (the density of the chi distribution with $m$ degrees of freedom), the chain rule for Clarke's subdifferential \cite[Theorem 2.3.9 (ii)]{clarke1983} provides the identity
\[
\partial^C_ue(\bar{u},v)=F'_\eta (\rho (\bar{u},v))\cdot\partial^C_u\rho(\bar{u},v)=\chi (\rho (\bar{u},v))\cdot\partial^C_u\rho(\bar{u},v).
\]
From here, the assertion of our lemma will follow along with \eqref{almostdone}, once we can show that $\tilde{M}(\bar{u},v)=M^*(\bar{u},v)$, where the latter set has been introduced in the statement of this lemma. Indeed, if $x\in\tilde{M}(\bar{u},v)$, then $x\in\mathcal{K}\subseteq\mathcal{O}$ and $\gamma (\bar{u},x)=\rho (\bar{u},v)$ by \eqref{kapparho}. Then, $x\in M^*(\bar{u},v)$ by definition of $\gamma$ and of the expected state $\bar{y}$. Conversely, let $x\in M^*(\bar{u},v)$. Then, $x\in\mathcal{O}$ and 
$\rho (\bar{u},v)=\gamma (\bar{u},x)$ again by definition of $\gamma$. Hence, $x\in M(\bar{u})$ by \eqref{mudef}. Now, \eqref{uniformbound} yields that $x\in\mathcal{K}$ and, hence, $x\in\tilde{M}(\bar{u},v)$ by \eqref{kapparho}.
\end{proof}
\subsection{Subdifferential of the probability function and optimality conditions in the concrete PDE setting}
\label{subsec:subdiff-prob-func-PDE}
The technical results of the previous section allow us now to formulate a fully explicit (in terms of the problem data) upper estimate for the subdifferential of the probability function in problem \eqref{eq:probuniform-problem} and explicit necessary optimality conditions for that same problem. Combining Corollary \ref{cor-theo} with Lemmas \ref{techlemma1} and \ref{techlemma2}, we end up with { the following theorem:}
\begin{theorem}\label{explicitprobderiv}
Fix some $\bar{u}\in L^2(D)$ with associated expected state $\bar{y}$ satisfying \eqref{sp1}. Then, $\varphi$ is locally Lipschitzian around $\bar{u}$ and, with the notation introduced in Lemma \ref{techlemma2}, the  subdifferential of the probability function at $\bar{u}$ can be estimated from above {as follows:} 
\begin{eqnarray}\label{explicitupperestimate}
\partial^C\varphi (\bar{u})\subseteq -\int\limits_{\{v\in\mathbb{S}^{m-1}\mid\rho (\bar{u},v)<\infty \}}\chi (\rho (\bar{u},v))\cdot{\rm clco}\,\left\{
\frac{1}{[P(0,\Sigma^{1/2}v)](x)}\cdot u_x \,\Big\vert\, x\in M^*(\bar{u},v)\right\}d\mu_\zeta(v).
\end{eqnarray}
Here, we use the notation $M^*(u,v)$  introduced in Lemma \ref{techlemma2}. The inclusion holds as an equality provided that $M^*(\bar{u},v)$ is a singleton for $\mu_\zeta$ a.e. $v\in\mathbb{S}^{m-1}$ with $\rho (\bar{u},v)<\infty$  (implying condition \eqref{measurezero2} thanks to Lemmas \ref{techlemma1} and \ref{techlemma2} and, hence, strict differentiability of $\varphi $) or that the set \eqref{bounded-realizations} is bounded (implying the assumptions of Lemma \ref{clarkereg}).
\end{theorem}

\noindent
Likewise, Corollary \ref{coroptcond} yields the following necessary optimality conditions: 
\begin{theorem}\label{theooptcond}
In the optimization problem \eqref{eq:probuniform-problem}, let there exist some control $\hat{u}\in L^2(D)$ such that the associated random state $\hat{y}$ solving \eqref{eq:probuniform-problem-b}--\eqref{eq:probuniform-problem-c} satisfies the probabilistic constraint \eqref{eq:probuniform-problem-d}
strictly, i.e.,
\begin{equation}\label{gsp}
\varphi(\hat{u})=\mathbb{P}(\hat{y}(x,\omega)\leq \alpha\quad\forall x\in D)>p.
\end{equation}
Consider some $u^*\in L^2(D)$ that is feasible in 
\eqref{eq:probuniform-problem}. Assume that the associated expected state $\tilde{y}:=S(u^*,0)$ satisfies
\begin{equation}\label{sp5}
\tilde{y}(x)<\alpha\quad\forall x\in\bar{D}.
\end{equation}
Then, for $u^*$ to be a solution to \eqref{optprob}, it is necessary that there exists some $\lambda\geq 0$ with
\begin{equation}\label{kkt4}
D F(u^*)\in -\lambda\int\limits_{\{v\in\mathbb{S}^{m-1}\mid\rho (u^*,v)<\infty\}}\chi (\rho (u^*,v))\cdot{\rm clco}\,\left\{
\frac{1}{[P(0,\Sigma^{1/2}v)](x)}\cdot u_x \, \Big\vert \, x\in M^*(u^*,v)\right\}d\mu_\zeta(v)
\end{equation}
and 
\begin{equation}\label{complementarity}
\lambda (\mathbb{P}(y^*(x,\omega)\leq \alpha\quad\forall x\in D)-p)=0,
\end{equation}
where $y^*(\cdot,\omega)=S(u^*,\xi (\omega))$ is the random state associated with the control $u^*$ and random event $\omega\in\Omega$. Moreover, this inclusion also becomes sufficient for $u^*$ to be a solution to \eqref{optprob}, if the conditions from Theorem \ref{explicitprobderiv} for \eqref{explicitupperestimate} to hold as an equality are satisfied.
\end{theorem}

\bigskip\noindent
The following example shall illustrate Theorems \ref{probderiv} and \ref{explicitprobderiv}, respectively. The data will be chosen in such as way so that \eqref{sp1} is satisfied and, hence, to apply both results. As a consequence, the probability function will be locally Lipschitzian. However, it will fail to be differentiable despite the fact that all problem data are smooth. This underlines why it is necessary to employ tools from generalized differential calculus in the analysis of the optimization problem \eqref{eq:probuniform-problem}. Moreover, we are going to calculate explicitly the upper estimate \eqref{explicitupperestimate} for the subdifferential of the probability function provided in Theorem \ref{explicitprobderiv}. 
\begin{example}\label{illustex}
In our problem \eqref{eq:probuniform-problem}, let
$m:=1, D:=(0,2\pi)^2, \alpha :=1, f_0:=0, \xi\sim\mathcal{N}(0,1)$ and
\[
\phi_1(x_1,x_2):= 2\sin^2x_1(\sin^2x_2-\cos^2x_2)+2\sin^2x_2(\sin^2x_1-\cos^2x_1). 
\]
Fix a control $\hat{u}(x_1,x_2):=\sin x_1\sin x_2$. Then, for each $\tau\in\mathbb{R}$ and each $z_1\in\mathbb{R}$, 
\[
y_{\tau ,z_1}(x_1,x_2):=\tau\sin x_1\sin x_2+z_1\sin^2x_1\sin^2x_2
\]
is the solution to the PDE
\[
-\Delta y(x)=2\tau\hat{u}(x)+f(x,z_1)\quad (x\in D),\quad y(x)=0\quad (x\in \partial D).
\]
For the purpose of illustration, we consider the restriction of the probability function $\varphi$ to the one-dimensional subspace { spanned by $\hat{u}$}, i.e., the probability function $\tilde{\varphi}\colon \mathbb{R}\to\mathbb{R}$ defined by 
\[
\tilde{\varphi}(\tau):=\varphi (2\tau\hat{u})=\mathbb{P}(\omega\mid y_{\tau ,\xi (\omega)}(x)\leq 1\quad\forall x\in (0,2\pi)^2).
\]
One easily checks that \eqref{sp1} is satisfied at $2\tau\hat{u}$ for all $\tau\in (-1,1)$. Accordingly, Corollary \ref{cor-theo} (via Theorem \ref{probderiv}) yields that, at all $2\tau\bar{u}$ with $\tau\in (-1,1)$, $\varphi$ is locally Lipschitzian (hence $\tilde{\varphi}$ is locally Lipschitzian at all  $\tau\in (-1,1)$) and the upper estimate of Theorem \ref{explicitprobderiv} holds true. 
We will derive an explicit formula for $\tilde{\varphi}$. Since $y_{\tau ,z_1}(x_1,x_2)=0$  whenever $\sin x_1\sin x_2=0$, one gets that
\[
\{z_1\mid y_{\tau ,z_1}(x)\leq\alpha\}=\left\{z_1\left|z_1\leq\frac{1-\tau\sin x_1\sin x_2}{\sin^2x_1\sin^2x_2}\quad\forall x\in (0,2\pi)^2:\sin x_1\sin x_2\neq 0\right.\right\}.
\]
If $0\leq\tau <1$, then the fraction on the right-hand side is minimized whenever $\sin x_1\sin x_2=1$  (e.g., by choosing $x:=(\pi/2,\pi/2$)). The achieved minimum equals $1-\tau$. Likewise, if $-1<\tau\leq 0$ then this fraction is minimized whenever $\sin x_1\sin x_2=-1$  (e.g., by choosing $x:=(\pi/2,3\pi/2)$). The achieved minimum equals $1+\tau$. Accordingly,
\[
\tilde{\varphi}(\tau)=\left\{\begin{array}{lll}
 \mathbb{P}(\omega\mid\xi (\omega )\leq 1+\tau)=\Phi (1+\tau)   &\mbox{if}&-1<\tau\leq 0 \\
 \mathbb{P}(\omega\mid\xi (\omega )\leq 1-\tau)=\Phi (1-\tau)    & \mbox{if}&0\leq\tau <1
\end{array}\right.
\]
Here, $\Phi$ denotes the distribution function of the one-dimensional standard Gaussian distribution $\mathcal{N}(0,1)$. 
It is locally Lipschitzian as expected, but it fails to be differentiable at zero because the derivatives at zero of the two pieces become $F(1)>0$ and $-F(1)<0$, respectively, where $F=\Phi'$ refers to the standard Gaussian density.

Next, we are going to calculate explicitly the upper estimate \eqref{explicitupperestimate} from Theorem \ref{explicitprobderiv} at the reference point $\bar{u}=0$. First, we note that $\Sigma^{1/2}=1$ due to $\xi\sim\mathcal{N}(0,1)$ and $\mathbb{S}^{m-1}=\mathbb{S}^0=\{-1,1\}$. By definition of the operator $P$ and of the state $y_{\tau ,z_1}$, we have that
\[
[P(0,1)](x)=y_{0 ,1}(x_1,x_2)=\sin^2x_1\sin^2x_2;\quad
[P(0,-1)](x)=y_{0 ,-1}(x_1,x_2)=-\sin^2x_1\sin^2x_2\quad\forall
(x_1,x_2)\in (0,2\pi)^2.
\]
Since $y_0=0$ by our data assumption $f_0=0$, we infer from \eqref{firstrep} that
\begin{eqnarray*}
\left\{r\geq 0\left|\max\limits_{x\in\bar{D}}\,\,[P(\bar{u},r\Sigma^{1/2}\cdot (-1))](x)+y_0(x)\leq\alpha\right.\right\}&=&
\left\{r\geq 0\left|\max\limits_{x\in [0,2\pi]^2}\,\,[P(0,-r)](x)\leq 1\right.\right\}\\&=&\left\{r\geq 0\left|\max\limits_{x\in [0,2\pi]^2}\,\,-r\sin^2x_1\sin^2x_2\leq 1\right.\right\}=[0,\infty ),
\end{eqnarray*}
hence $\rho (\bar{u},-1)=\infty$ by definition of $\rho$. Therefore, $v=-1$ does not contribute to the integral in \eqref{explicitupperestimate}. Likewise,
\[
\left\{r\geq 0\left|\max\limits_{x\in\bar{D}}\,\,[P(\bar{u},r\Sigma^{1/2}\cdot 1)](x)+y_0(x)\leq\alpha\right.\right\}=\left\{r\geq 0\left|\max\limits_{x\in [0,2\pi]^2}\,\,r\sin^2x_1\sin^2x_2\leq 1\right.\right\}=[0,1],
\]
hence $\rho (\bar{u},1)=1<\infty$ and so $v=1$ does contribute to the integral in \eqref{explicitupperestimate}. Since the uniform distribution on $\mathbb{S}^0$ can be represented as $\mu_\zeta=\frac{1}{2}\delta_{-1}+\frac{1}{2}\delta_{1}$, where $\delta_t$ is the Dirac measure at $t\in\mathbb{R}$, the negative value of this integral appearing in \eqref{explicitupperestimate} reduces to
\[
-\frac{1}{2}\chi (\rho (0,1))\cdot{\rm clco}\,\left\{
\frac{1}{[P(0,1)](x)}\cdot u_x\mid x\in M^*(0,1)\right\}=
-\frac{1}{2}\chi (1)\cdot{\rm clco}\,\left\{
\frac{1}{\sin^2x_1\sin^2x_2}\cdot u_x\mid x\in M^*(0,1)\right\}.
\]
It remains to determine the index set $M^*(0,1)$. First note that $\bar{y}=0$ holds true for the expected state (associated with control $\bar{u}=0$ and $z_1=0$) by our data assumption $f_0=0$. By definition of $M^*$ in Lemma \ref{techlemma2} we arrive at
\[
M^*(0,1)=\{x\in [0,2\pi]^2\mid\sin^2x_1\sin^2x_2=1\}=\{(\pi/2,\pi/2),(3\pi/2,\pi/2),(\pi/2,3\pi/2),(3\pi/2,3\pi/2)\}.
\]
This yields the final explicit upper estimate for the subdifferential of the probability function
\[
\partial^C\varphi (0)\subseteq -\frac{1}{2}\chi (1)\cdot{\rm co}\,\left\{u_{(\pi/2,\pi/2)},u_{(3\pi/2,\pi/2)},u_{(\pi/2,3\pi/2)},u_{(3\pi/2,3\pi/2)}\right\},
\]
where the $u_x\in L^2((0,2\pi)^2)$ have been introduced in the statement of Lemma \ref{techlemma2} and where the closure operation in the convex hull can be omitted due to the set being finite.
\end{example}

\noindent
We want to illustrate Theorem \ref{theooptcond} in an example.
\begin{example}\label{illustex2}
In our problem \eqref{eq:probuniform-problem}, let (with $\Phi$ being the distribution function of $\mathcal{N}(0,1)$)
\[
 D:=(0,1)^2, m:=1, \alpha :=1/16, f_0:=0, \xi\sim\mathcal{N}(0,1), \phi_1(x_1,x_2):= 2(x_1 (1-x_1)+x_2(1-x_2)), p:=\Phi(1)\approx 0.841. 
\]
Fix a control $\tilde{u}\in L^2(D)$ defined by the relation 
\begin{equation}\label{utilde}
 (\tilde{u},h)_{L^2(D)} =[P(h,0)](1/2,1/2)=[A^{-1}h](1/2,1/2)\quad\forall h\in L^2(D).
\end{equation}
We define the objective in \eqref{eq:probuniform-problem-a} by $F(u):=\norm{u-\tilde{u}}{L^2(D)}^2$ and claim that $u^*:=0$ is a solution of the optimization problem.  Evidently, for each $z_1\in\mathbb{R}$, the function
$y(x_1,x_2):=z_1x_1(1-x_1)x_2(1-x_2)$
is a solution to the PDE
\begin{equation}\label{justthepde2}
-\Delta y(x)=z_1\phi_1(x)\quad (x\in D),\quad y(x)=0\quad (x\in \partial D),
\end{equation}
which is \eqref{eq:PDE} with $u=0$. It follows from \eqref{affinedecomp} that
\begin{equation}\label{sandp}
[S(0,z_1)](x)=[P(0,z_1)](x)=[P(z_1\phi_1,0)](x)=z_1x_1(1-x_1)x_2(1-x_2)\quad\forall z_1\in\mathbb{R}.
\end{equation}
In particular, $S[(u^*,0)](x)=0$ for all $x\in\bar{D}$ so that \eqref{sp5} is satisfied.
Moreover,
\[
\varphi (u^*)=\varphi (0)=
\mathbb{P}(\omega \mid\xi (\omega )x_1(1-x_1)x_2(1-x_2)\leq 1/16\,\,\forall x\in (0,1)^2)=\mathbb{P}(\omega\mid\xi (\omega )\in (-\infty ,1])=\Phi (1)=p.
\]
This means that $u^*=0$ is a feasible control in problem \eqref{eq:probuniform-problem} at which the probabilistic constraint is active.
Next, we calculate the Clarke subdifferential $\partial^C\varphi (u^*)$.
Recalling that, due to the assumptions on our data, $\Sigma =1$ and $\mathbb{S}^{m-1}=\mathbb{S}^0=\{-1,1\}$, we identify the two radii $\rho (0,1), \rho (0,-1)$ via \eqref{firstrep} from the relations 
\begin{eqnarray*}
[0,\rho (0,1)]&=&\left\{r\geq 0\left|\max\limits_{x\in [0,1]^2}\,\,[P(0,r)](x)\leq 1/16\right.\right\}=[0,1],\\
\left[0,\rho (0,-1)\right]&=&\left\{r\geq 0\left|\max\limits_{x\in [0,1]^2}\,\,[P(0,-r)](x)\leq 1/16\right.\right\}=[0,\infty ).
\end{eqnarray*}
Accordingly, $\rho (0,1)=1$ and $\rho (0,-1)=\infty$.
By definition and by virtue of \eqref{sandp},
\[
M^*(0,1)=\{x\in [0,1]^2\mid [P(0,1)](x)>0,\,\,[P(0,1)](x)=1/16\}=\{(1/2,1/2)\}.
\]
Consequently, $\# M^*(0,1)=1$ and, hence, Theorem \ref{explicitprobderiv} yields that
\begin{equation}\label{gradformula}
\partial^C\varphi (u^*)=-\int\limits_{v=1}
\frac{\chi (\rho (0,1))}{[P(0,1)](1/2,1/2)}\cdot u_{(1/2,1/2)}d\mu_\zeta(v)=\left\{-\frac{1}{2}\frac{\chi (1)}{[P(0,1)](1/2,1/2)}\cdot u_{(1/2,1/2)}\right\}=\{-8\chi (1)\tilde{u}\},
\end{equation}
where $\tilde{u}$ has been introduced above. In particular, $\varphi$ is strictly differentiable with derivative $\nabla\varphi (u^*)=-8\chi (1)\tilde{u}$.

Finally, we want to apply the optimality conditions of Theorem \ref{theooptcond}. In order to do so, it remains first to check condition \eqref{gsp}. Thanks to \eqref{sandp} we have that $[P(\phi_1,0)](1/2,1/2)=1/16>0$.
This implies that $\tilde{u}\neq 0$ because otherwise a contradiction with \eqref{utilde} would occur. Now, assuming that condition \eqref{gsp} is violated, it would follow that $\varphi (u)\leq p$ for all $u\in L^2((0,1)^2)$. Owing to $\varphi (u^*)=p$ (see above), it would follow that $u^*$ is a (global) maximum of $\varphi$. By Clarke's necessary optimality condition \cite[Proposition 2.3.2]{clarke1983} and by \eqref{gradformula} we would end up at the contradiction
$0\in\partial^C(\varphi )(u^*)=\{-8\chi(1)\tilde{u}\}$
with $\tilde{u}\neq 0$ (see above). Summarizing, conditions \eqref{gsp}, \eqref{sp5}, and $\# M^*(0,1)=1$ (while $\rho (0,-1)=\infty$) are satisfied, which allows us to apply Theorem \ref{theooptcond}. In order to prove our claim that $u^*:=0$ is a solution of \eqref{eq:probuniform-problem}, it remains to verify \eqref{kkt4} and \eqref{complementarity}. We define a multiplier $\lambda :=\frac{-1}{4\chi (1)}<0$. As for the complementarity relation, it is satisfied thanks to the already stated equality $\varphi (0)=p$. On the other hand, \eqref{gradformula} yields that
\begin{eqnarray*}
\nabla F(0)&=&-2\tilde{u}=\lambda 8\chi (1)\tilde{u} =\lambda\partial^C\varphi (u^*)
=\lambda\int\limits_{v=1}
\frac{\chi (\rho (0,1))}{[P(0,1)](1/2,1/2)}\cdot u_{(1/2,1/2)}d\mu_\zeta (v)\\
&=&\lambda\int\limits_{\{v\in\mathbb{S}^{0}\mid\rho (u^*,v)<\infty\}}\chi (\rho (u^*,v))\cdot{\rm clco}\,\left\{
\frac{1}{[P(0,\Sigma^{1/2}v)](x)}\cdot u_x \, \Big\vert \, x\in M^*(u^*,v)\right\}d\mu_\zeta(v)
\end{eqnarray*}
where we recall that $\rho (0,-1)=\infty$ and $M^*(0,1)=\{(1/2,1/2)\}$. Now, \eqref{kkt4} as a sufficient condition proves that $u^*=0$ is a solution of the formulated optimization problem.
\end{example}
\section{Almost sure and robust constraints}
\label{sec:almost-sure}
In Section \ref{sec:prob-constraints}, we considered optimization problems with chance constraints of the form $\mathbb{P}(g(u,\xi (\omega))\leq 0)\geq p$. In this section, we will consider two closely related models:
\begin{eqnarray*}
g(u,\xi (\omega))\leq 0\quad \mathbb{P}\mbox{-almost surely}&&\mbox{(almost sure)},\\
g(u,z)\leq 0\quad\forall z\in\mbox{ supp }\xi&&\mbox{(robust)}.
\end{eqnarray*}
Optimality conditions for PDE-constrained optimization problems with almost sure state constraints were recently been derived in \cite{Geiersbach2021c, Geiersbach2022a} and robust formulations have also been seen in the literature; see, e.g., \cite{Kolvenbach2018}. { In \cite{Ramponi2018}, these models were compared for convex optimization problems with a finite-dimensional control.} We will show optimality conditions for the almost sure and robust cases (Section~\ref{subsec:abstract-approach} and Section~\ref{sec:robust-constraints}, respectively) using our model PDE. First, we will discuss some  nuances in the various models.

The probabilistic model from Section~\ref{sec:prob-constraints} is equivalent to the almost sure model if $p=1.$ However, as mentioned in Remark~\ref{rem:case-p-equals-1}, the optimality conditions obtained there exclude the case $p=1.$ One might hope that optimality conditions for the probability level $p=1$ could be obtained in the limit $p\to 1$, but this is not the case, as demonstrated in the following example.
\begin{example}\label{asvsprob}
Let $\xi$ be a one-dimensional random variable with density $f(x):=2(1-x)\chi_{[0,1]}$ and consider the problem
\begin{equation}\label{exprob}
\min_{u \in \R} \{F(u):=u\}\quad\textup{s.t.}\quad \varphi(u):=\mathbb{P}(g(u,\xi (\omega))\leq 0) \geq p,
\end{equation}
where $g(u,z):=z-u$ for $z,u\in\mathbb{R}$. Then, the probability function is given by
\[
\varphi (u)=\begin{cases}
    0 & u<0, \\
   2u-u^2  & u\in [0,1],\\1&u>1.
\end{cases}
\]
At $u>0$, $\varphi$ is continuously differentiable. For $p>0$,  the constraint set $\varphi (u)\geq p$ is identified as the interval $[u^*_p,\infty )$ with $u^*_p:=1-\sqrt{1-p}\in (0,1]$ necessarily being the solution of \eqref{exprob}. Note that $\varphi (u^*_p)=p$, so the inequality constraint is binding at $u^*_p$. If $p=1$, then $\varphi'(u^*_p)=\varphi'(1)=0$. Hence, the basic constraint qualification is violated at the solution $u^*_1=1$ of problem \eqref{exprob} 
and no optimality condition can be established. If $p<1$,
then
\[
\varphi'(u^*_p)=2-2(1-\sqrt{1-p})>0
\]
and the basic constraint qualification for necessary optimality conditions is satisfied. It follows that the Lagrange multiplier of the KKT conditions is uniquely determined by the relation
\[
1=F'(u^*_p)=\lambda_p\varphi'(u^*_p)
\]
yielding $\lambda_p=1/(2\sqrt{1-p})$. Now, for $p\to 1$, we get
$F'(u^*_p)\to 1$ on the one hand, but $\varphi'(u^*_p)\to 0$ and $\lambda_p\to\infty$ on the other hand, so that the limit of KKT conditions does not provide any information.  
\end{example}
The negative observations made in the previous example result from the fact that setting $p=1$ in the probabilistic version provides a degenerate problem description. In contrast, the almost sure and robust perspective (which are equivalent in this example) both work well:
\begin{example}
We reconsider problem \eqref{exprob} but in its equivalent robust description
\begin{equation}
\min_{u \in \R} \{F(u):=u\}\quad\textup{s.t.}\quad \tilde{\varphi}(u):=\sup\limits_{z\in [0,1]} g(u,z) \leq 0
\end{equation}
(note that $[0,1]$ is the support of $\xi$ in Example~\ref{asvsprob}). The fact that this problem is equivalent to \eqref{exprob} will follow from Lemma \ref{asrob} below, because $g$ is continuous in $z$, and hence the robust description is equivalent to the almost-sure description, which in turn is always equivalent with the probabilistic one. Clearly, the sup-function is given by
\[
\tilde{\varphi}(u)=\sup\limits_{z\in [0,1]}z-u=1-u.
\]
Then, $\tilde{\varphi}(2)<0$, so that $\hat{u}=2$ is a Slater point of the problem and equivalent optimality conditions can be formulated. The feasible set equals the interval $[1,\infty)$ and at the solution $u^*=1$ of the problem, one has the KKT relation 
\[
1=F'(1)=(-1)(\tilde{\varphi})'(1)=(-1)(-1).
\]
\end{example}
The following example shows that generally, an almost sure model is not equivalent to a robust one.
\begin{example}\label{robvsas}
Consider the one-dimensional control problem
\[
\min_{u \in L^2(0,1)} \norm{u}{L^2(0,1)}^2 \quad \textup{s.t.} \quad y(x,\omega)\leq 1\quad\forall x\in [0,1]\quad\mathbb{P}\textup{-a.s.},
\]
where $y(x,\omega)$ is the solution to
\[
-\Delta y(x,\omega)=u(x)+2f(\xi (\omega))\quad x\in (0,1) \,\,\mathbb{P} \textup{-a.s.};\quad y(x,\omega)=0\quad x\in \{0,1 \} \,\,\mathbb{P}\textup{-a.s.},
\]
$\xi$ is a one-dimensional random variable uniformly distributed on $[0,1]$, and 
\[
f(z):=\begin{cases} 4&z\neq 0.5,\\5&z=0.5.
\end{cases}
\]
Then, $u:=0$ is a feasible control of the problem. Indeed, the random states associated with this control are given by
\[
y(x,\omega)=f(\xi (\omega))x(1-x)\quad\forall x\in [0,1], \forall \omega\in\Omega .
\]
It follows that 
\[
y(x,\omega)\leq\frac{1}{4}f(\xi (\omega))\leq 1\quad\forall\omega : \xi (\omega )\neq 0.5.
\]
Since $\xi (\omega )\neq 0.5$ $\mathbb{P}$-almost surely, feasibility of $u=0$ follows. In particular, this control is optimal in the given problem. Now, adopting the robust perspective, we observe that $[0,1]$ is the support of the random variable and formulate the problem 
\[
\min \norm{u}{L^2(0,1)}^2 \quad \textup{s.t.} \quad y(x,z)\leq 1\quad\forall x\in [0,1]\quad\forall z\in [0,1],
\]
where $y(x,z)$ is the solution to
\[
-\Delta y(x,z)=u(x)+2f(z)\quad x\in (0,1), z\in [0,1];\quad y(x,z)=0\quad x\in \{0,1 \}, z\in [0,1].
\]
Then, the parameterized state $y(x,z)=f(z)x(1-x)$ solving the PDE under the control $u=0$ fails to be feasible because of $y(0.5,0.5)=5/4>1$. Consequently, $u=0$ is not feasible, so cannot be optimal for the robust problem.
\end{example}

{ It is notable that an example is constructed by Ramponi in \cite[Example 3]{Ramponi2018} that is similar in spirit to our Example~\ref{robvsas}.}
The reason for the difference between the robust and almost sure perspectives in { these examples is due to a lack of lower semicontinuity} with respect to the parameter $z$. { It is interesting to observe that in \cite{Ramponi2018}, the robust solution is generally inconsistent in the following sense: the solution to the problem generated by randomly sampled scenarios does not converge almost surely to the solution of the robust problem as the number of scenarios is taken to infinity.}

We have the following relation between models { with almost sure or robust constraints}.
\begin{lemma}\label{asrob}
Let $\xi$ be an $m$-dimensional random vector {with density} on a probability space $(\Omega,\mathcal{F},\mathbb{P})$ and denote by $\Xi\subseteq\mathbb{R}^m$ the support of its law $\mathbb{P}\circ\xi^{-1}$. Furthermore, let $h\colon\mathbb{R}^m\to\mathbb{R}$ be a lower semicontinuous function. Then,
\[
h(\xi (\omega))\leq 0\quad\mathbb{P}\mbox{-a.s.}\quad \Longleftrightarrow \quad h(z)\leq 0\quad\forall z\in\Xi.
\]
\end{lemma}
\begin{proof}
 ($\Longrightarrow$): Assume there exists $z\in\Xi$ with $h(z)>0$. Since $h$ is lower semicontinuous, there exists some $r>0$ such that $h(z')>0$ for all $z'$ in the ball $\mathbb{B}(z,r) \in \R^m$ of radius $r$ centered around the point $z$. From here and from $z\in\Xi$, we get the contradiction
 \[
 \mathbb{P}\{\omega\mid h(\xi (\omega))>0\}\geq \mathbb{P}\{\omega\mid\xi (\omega)\in\mathbb{B}(z,r)\}>0.
 \]
 ($\Longleftarrow$): The reverse direction follows readily from
 \[
 1=\mathbb{P}\{\omega\mid\xi (\omega)\in\Xi\}\leq\mathbb{P}\{\omega\mid h(\xi (\omega))\leq 0\}.
 \]
\end{proof}

\subsection{Almost sure state constraints}
\label{subsec:abstract-approach}
In this section, we will derive optimality conditions for a model problem with almost sure state constraints. As mentioned in the introduction, the functional-analytic techniques applied here largely come from traditional approaches developed by Rockafellar and Wets for two-stage problems in stochastic programming \cite{Rockafellar1976,Rockafellar1976b,Rockafellar1975,Rockafellar1976c}. Much of the analysis there hinged on an assumption of \textit{relatively complete recourse}, a requirement that every control $u$ must result in a feasible state $y$. This assumption is too strong for our model problem \eqref{eq:a.s.-problem} below, since this would make the state constraint redundant for any choice of $u.$ In \cite{Rockafellar1976c} and in \cite{Geiersbach2021c}, this assumption was lifted, resulting in singular Lagrange multipliers. Unlike in the previous section, there is little hope for a more explicit representation and these elements do not provide a structure that is amenable to numerical approximation. In practice, a regularization as in \cite{Geiersbach2021c} can be used so that the singular terms only appear in the limit. Here, however, we study the unregularized problem, which also shares structural similarities to that of its robust counterpart, as we will see in Section~\ref{sec:robust-constraints}. As a main result, we will see that we can more efficiently arrive at the conditions provided by \cite{Geiersbach2021c, Geiersbach2022a} by first applying standard optimality theory and then refining the Lagrange multipliers using a Yosida--Hewitt-type decomposition. 

{ We recall the setting given in Assumption~\ref{subasu:U}--Assumption~\ref{subasu:F} and consider} the problem
\begin{subequations}
    \label{eq:a.s.-problem}
    \begin{alignat}{3}
    \min_{u \in L^2(D)}\, F(u) & & &&  \\
   \text{s.t.}  \quad -\Delta  y(x,\omega)  &=  u(x) + \tilde{f}(x,\omega), & &&\quad x \in D \quad \text{$\mathbb{P}$-a.s.}, \label{eq:a.s.-problem-constraint1} \\
 y(x,\omega) &=0, & &&\quad x \in \partial D  \quad \text{$\mathbb{P}$-a.s.}, \label{eq:a.s.-problem-constraint2}\\
 y(x,\omega) &\leq \alpha & && x\in D \quad \text{$\mathbb{P}$-a.s.}\phantom{.} \label{eq:a.s.-problem-constraint3}
    \end{alignat}
\end{subequations}
We have replaced the structured source term $f$ with a random field $\tilde{f}$; while more general differential operators would be possible using the results from \cite{Geiersbach2021c, Geiersbach2022a}, we choose the simple model here for the sake of comparison with the model introduced in Section~\ref{subsec:prob-PDE}. We will now assume the following.

\begin{assumption}
\label{ass:PDE-standing-2}
The open and bounded set $D\subseteq\mathbb{R}^d$ ($d=1,2,3$) is of class $S$. 
%The open and bounded set $D\subseteq\mathbb{R}^d$ ($d=1,2,3$) is either polyhedrally convex or has a $C^{1,1}$-boundary.  
Additionally, $u \in L^2(D)$ and $\tilde{f}\in L_{\pP}^\infty(\Omega,L^2(D))$.
\end{assumption}

%\noindent Under this assumption, we have the following result from~\cite[Proposition 1]{Gahururu2022}.
%\noindent For the following result, we use $Y:=\mathcal{C}(\bar{D}).$
\noindent We recall the continuous linear operator $A^{-1} \colon L^2(D) \rightarrow \mathcal{C}(\bar{D})$ defined after Lemma~\ref{lem:cont-dep-sol}.
\begin{lemma}
\label{lem:cont-dep-sol-2}
Suppose Assumption~\ref{ass:PDE-standing-2} holds. Then for any $u, \tilde{f}(\cdot,\omega)\in L^2(D)$, and $\pP$-almost every $\omega \in \Omega$, there a unique solution $y(\cdot,\omega) \in H_0^1(D)\cap \mathcal{C}(\bar{D})$ to \eqref{eq:a.s.-problem-constraint1}--\eqref{eq:a.s.-problem-constraint2}. Moreover, $y \in L_{\pP}^\infty(\Omega, \mathcal{C}(\bar{D}))$ and there exists a constant $C>0$ such that 
\begin{equation}
\label{eq:apriori-PDE-omega}
     \norm{y}{L_{\pP}^\infty(\Omega,\mathcal{C}(\bar{D}))}   \leq C( \|u\|_{L^2(D)}+\|\tilde{f}\|_{L_{\pP}^\infty(\Omega,L^2(D))}).
\end{equation}
\end{lemma}
\begin{proof}
Following the proof of Lemma~\ref{lem:cont-dep-sol}, we obtain the existence of a unique solution $y(\cdot,\omega) \in H_0^1(D) \cap \mathcal{C}(\bar{D})$ and a $C>0$ such that
\[
\norm{y(\cdot,\omega)}{\mathcal{C}(\bar{D})} \leq C \lVert u+\tilde{f}(\cdot,\omega)\rVert_{L^2(D)} \leq C(\lVert u \rVert_{L^2(D)} + \lVert \tilde{f}\rVert_{L_{\pP}^\infty(\Omega,L^2(D))}),
\]
which shows \eqref{eq:apriori-PDE-omega} due to our assumption on $\tilde{f}.$ Moreover, $\omega \mapsto y(\cdot,\omega)$ is measurable since $\tilde{f}$ is measurable and the operator $A^{-1}\colon L^2(D) \rightarrow \mathcal{C}(\bar{D})$ is continuous. 
\end{proof}
 
Analogously to \eqref{affinedecomp}, we will define a control-to-state operator $\mathcal{S}\colon L^2(D) \rightarrow L_{\pP}^\infty(\Omega, \mathcal{C}(\bar{D}))$ by 
\[
\mathcal{S}(u) = \mathcal{A}^{-1}(\mathcal{B}u + \tilde{f}).
\]
Here, $\mathcal{A}^{-1}\colon  L_{\pP}^\infty(\Omega, L^2(D)) \rightarrow L_{\pP}^\infty(\Omega, \mathcal{C}(\bar{D}))$ is defined by $[\mathcal{A}^{-1}y](\omega)=A^{-1}y(\cdot,\omega)$ a.s. The operator
$\mathcal{B}$ is the continuous injection that maps each element of $L^2(D)$ to itself in $L_{\pP}^\infty(\Omega, L^2(D))$. 
For the state constraint, let $\boldsymbol{\alpha}$ denote the constant counterpart to $\alpha$ in $\mathcal{C}(\bar{D})$ and define
\[\mathcal{G}\colon L^2(D) \rightarrow L_{\pP}^\infty(\Omega,\mathcal{C}(\bar{D})), u \mapsto \mathcal{S}(u)-\boldsymbol{\alpha}.
\]
Additionally, we define the (convex and closed) cone $\mathcal{K}:=\{y \in L_{\pP}^\infty(\Omega,\mathcal{C}(\bar{D})) : y(x,\omega)\leq 0 \, \text{ a.e. in } D \text{ $\pP$-a.s.}\}$ and its polar cone $\mathcal{K}^{-}:=\{ y^* \in (L_{\pP}^\infty(\Omega,\mathcal{C}(\bar{D})))^*: \langle y^*,y\rangle\leq 0 \quad \forall y \in \mathcal{K}\}.$ It is straightforward to verify the $(-\mathcal{K})$-convexity of $\mathcal{G}$, i.e., 
\[
\lambda \mathcal{G}(u)+(1-\lambda)\mathcal{G}(v) - \mathcal{G}(\lambda u+(1-\lambda)v) \in -\mathcal{K} \quad {\forall u, v \in L^2(D), \forall \lambda \in [0,1]},
\]
using the affine linearity of $\mathcal{G}.$ First, we present optimality conditions obtained using standard convex analysis.
\begin{lemma}
\label{lemma:optimality-almost-sure-basic}
Suppose Assumption~\ref{ass:PDE-standing-2} holds and $F$ is convex and continuously Fr\'echet differentiable on $L^2(D)$. Assume the constraint qualification $0 \in \textup{int} \{\mathcal{G}(L^2(D))-\mathcal{K} \}$ holds. Then, for $u^*$ to be a solution to \eqref{eq:a.s.-problem}, it is necessary and sufficient that there exists some multiplier $\lambda^* \in (L_{\pP}^\infty(\Omega,\mathcal{C}(\bar{D})))^*$ with 
\begin{subequations}
\label{eq:KKT-operator-short}
  \begin{align}
 & \langle D F(u^*), h\rangle + \langle \lambda^*,\mathcal{A}^{-1}\mathcal{B}h\rangle = 0 \quad  \forall h \in L^2(D), \label{eq:KKT-operator-1}\\
 & \mathcal{G}(u^*) \in \mathcal{K}, \quad \lambda^* \in  \mathcal{K}^{-}, \quad \langle \lambda^*, \mathcal{G}(u^*) \rangle = 0. \label{eq:KKT-operator-2}
  \end{align}
\end{subequations} 
\end{lemma}

\begin{proof}
Let $L(u,\lambda):=F(u)+\langle \lambda, \mathcal{G}(u)\rangle.$ Since $F$ is convex and $\mathcal{G}$ is $(-\mathcal{K})$-convex, the problem \eqref{eq:a.s.-problem} is convex. The constraint qualification implies that the set of Lagrange multipliers is nonempty (see \cite[Theorem 3.6]{Bonnans2013}). In particular, we have that there exists $\lambda^* \in (L_{\pP}^\infty(\Omega,\mathcal{C}(\bar{D})))^*$ such that
\begin{equation}
\label{eq:KKT-simple}
0 \in \partial_u L(u^*,\lambda^*) \quad \text{and} \quad \lambda^* \in N_{\mathcal{K}}(\mathcal{G}(u^*)),
\end{equation}
where $\partial_u$ denotes the partial subdifferential with respect to $u$ in the sense of convex analysis.
The operator $\mathcal{G}\colon L^2(D) \rightarrow L^\infty_{\pP}(\Omega, \mathcal{C}(\bar{D}))$ is affine linear and is bounded thanks to \eqref{eq:apriori-PDE-omega}. Therefore, it is Fr\'echet differentiable with $D\mathcal{G}(u)[h] = \mathcal{A}^{-1}\mathcal{B}h$ for all $h \in L^2(D).$
This means the first condition in \eqref{eq:KKT-simple} simplifies to  \eqref{eq:KKT-operator-1}. The conditions \eqref{eq:KKT-operator-2} follow from the fact that $\mathcal{K}$ is a convex cone; see \cite[p.~150]{Bonnans2013}. Conditions \eqref{eq:KKT-operator-short} are necessary and sufficient by \cite[Proposition 3.3]{Bonnans2013}.
\end{proof}

\begin{remark}
    The constraint qualification $0 \in \textup{int} \{\mathcal{G}(L^2(D))-\mathcal{K} \}$ means that for every $w$ in a neighborhood of zero in $L_{\pP}^\infty(\Omega,\mathcal{C}(\bar{D}))$ there exists a $u\in L^2(D)$ satisfying $\mathcal{G}(u) +w\in \mathcal{K}.$ In particular, there exists a $u$ such that $\mathcal{G}(u)(x,\omega) < 0$ a.e.~in $D$ $\pP$-a.s. If $y=\mathcal{S}(u)$, we obtain the existence of a state that is strictly feasible, i.e., $y(x,\omega) < \alpha$ a.e.~in $D$ $\pP$-a.s. Consider now the value function $v(w):=\inf\{F(u) \mid u \in L^2(D), \mathcal{G}(u)+w \in \mathcal{K} \}$. Then $0 \in \textup{int} \{\mathcal{G}(L^2(D))-\mathcal{K} \}$ is equivalent to continuity of $v$ at $0 \in L^2(D)$, i.e., $0 \in \textup{int dom } v.$ %Bonnans/Shapiro, p. 149
    This provides the connection to \cite[Assumption 3.5 (ii)]{Geiersbach2022a}. Note that \cite[Assumption 3.5 (i)]{Geiersbach2022a} ensures that a minimizer exists, something we do not require here.
\end{remark}

It is possible to refine these optimality conditions by taking advantage of the following Yosida--Hewitt-type decomposition result from \cite[Appendix 1, Theorem 3]{Ioffe1972} and \cite{Levin1974}. Let $X$ be a separable Banach space. A continuous linear functional $v \in (L_{\pP}^\infty(\Omega,X))^*$ of the form
\begin{equation}
\label{eq:absolutely-continuous-functionals}
  \langle v, x\rangle = \int_{\Omega} \langle x^*(\omega), x(\omega) \rangle \D \pP(\omega)
\end{equation}
for some { weakly measurable\footnote{ A vector-valued random variable $x^*\colon \Omega \rightarrow X^*$ is said to be weakly measurable if $\omega \mapsto \langle x^*(\omega), x\rangle$ is measurable for every $x \in X$.}  $x^*\colon \Omega \rightarrow  X^*$ with $\lVert x^* \rVert_{L^1_{\pP}(\Omega, X^*)}<\infty$} is called \textit{absolutely continuous}. We call $v^\circ \in (L_{\pP}^\infty(\Omega,X))^*$  \textit{singular (relative to $\pP$)} if there exists a sequence $\{F_n\} \subset  \mathcal{F}$ with $F_{n+1} \subset F_n$ for all $n$ and $\pP(F_n) \rightarrow 0$ as $n\rightarrow \infty$ such that $v^\circ$ is concentrated on $\{ F_n\}$, i.e., $ \langle v^\circ, x\rangle= 0$
for all $x \in L_{\pP}^\infty(\Omega, X)$, that vanish on some $F_n$. { For the following results, we define the sets $\Lambda_\circ(X^*):=\{ v \in (L_{\pP}^\infty(\Omega,X))^* \mid v \text{ is singular}\}$ and $\Lambda_1(X^*):=\{ x^* \colon \Omega \rightarrow X^* \mid x^* \text{ is weakly measurable and } \lVert x^* \rVert_{L_{\pP}^1(\Omega,X^*)}<\infty\}$.}

\begin{theorem}[Ioffe and Levin]
\label{thm:ioffe-levin}
Each functional $v^* \in (L_{\pP}^\infty(\Omega,X))^*$ has a unique decomposition 
\begin{equation*}
\label{eq:decomposition-dual-space}
v^* = v^a + v^\circ,
\end{equation*}
where $v^a$ is absolutely continuous, $v^\circ$ is singular relative to
$\pP$, and 
\begin{equation*}
  \lVert v^* \rVert_{(L_{\pP}^\infty(\Omega,X))^*} = \lVert v^a \rVert_{(L_{\pP}^\infty(\Omega,X))^*}+\lVert v^\circ \rVert_{(L_{\pP}^\infty(\Omega,X))^*}.
\end{equation*}
\end{theorem}

{ Due to this decomposition result, for any element $x \in L_{\pP}^\infty(\Omega, X),$ we have
\[
\langle v^*,x\rangle = \langle v^a,x\rangle + \langle v^\circ, x \rangle = \E[\langle v(\cdot),x(\cdot)\rangle]+\langle v^\circ,x\rangle,
\]
where $v \in \Lambda_1(X^*)$ is the element corresponding to $v^a.$ }
We now have the following result. 
\begin{theorem}
    \label{thm:OC-a.s.}
With the assumptions from Lemma~\ref{lemma:optimality-almost-sure-basic}, for $u^*$ to be a solution to \eqref{eq:a.s.-problem}, it is necessary and sufficient that there exist multipliers $\lambda \in { \Lambda_1(\mathcal{C}(\bar{D})^*)}$ and $\lambda^\circ \in { \Lambda_\circ(\mathcal{C}(\bar{D})^*)}$ with 
\begin{subequations}
\label{eq:KKT-operator-YH}
  \begin{align}
 &  D F(u^*) + \E[A^{-*}\lambda] + \mathcal{B}^*\mathcal{A}^{-*} \lambda^\circ = 0, \label{eq:KKT-YH-1}\\
  &A y^*(\cdot, \omega) - u^* - \tilde{f}(\cdot, \omega) = 0, \label{eq:KKT-YH-2} \\
 &  y^*(\cdot,\omega)\leq \alpha \text{ a.e.~in D}, \quad \lambda(\cdot,\omega) \geq 0 \text{ a.e.~in D }, \quad  \langle \lambda(\cdot,\omega), y^*(\cdot,\omega)-\boldsymbol{\alpha} \rangle = 0, \label{eq:KKT-YH-3} \\
& \lambda^\circ \in  \mathcal{K}^{-}, \quad \langle \lambda^\circ, y^*-\boldsymbol{\alpha}  \rangle = 0, \label{eq:KKT-YH-4}
  \end{align}
\end{subequations} 
with pointwise conditions (in $\omega$) holding $\pP$-a.s.
\end{theorem}
\begin{proof}
Fix $h \in L^2(D)$ and let the unique decomposition of $\lambda^*$ be given by $\lambda^*=\lambda^a+\lambda^\circ$. Note first that
$\langle \lambda^*,\mathcal{A}^{-1}\mathcal{B}h\rangle =  \langle \lambda^\circ,\mathcal{A}^{-1}\mathcal{B}h\rangle + \langle \lambda^a,\mathcal{A}^{-1}\mathcal{B}h\rangle$. Letting $\lambda \in { \Lambda_1(\mathcal{C}(\bar{D})^*)}$ be the integrable element identified with $\lambda^a$, and recalling that $\mathcal{B}$ is the identity in $L^2(D)$ for every $\omega$, we have 
\[
\langle \lambda^*,\mathcal{A}^{-1}\mathcal{B}h\rangle = \E[\langle \lambda,A^{-1} h\rangle] + \langle \lambda^\circ,\mathcal{A}^{-1}\mathcal{B}h\rangle = \E[\langle A^{-*}\lambda,h\rangle] + \langle \mathcal{B}^*\mathcal{A}^{-*}\lambda^\circ,h\rangle.
\]
Therefore, \eqref{eq:KKT-operator-1} is equivalent to \eqref{eq:KKT-YH-1}. Condition \eqref{eq:KKT-YH-2} is equivalent to setting $y^* = \mathcal{S}(u^*)$ and the first inequality in \eqref{eq:KKT-YH-3} is evidently equivalent to $\mathcal{G}(u^*) \in \mathcal{K}.$ 

We now observe the decomposition of the condition $\lambda^* \in \mathcal{K}^-$ and claim that since $\langle \lambda^*, y\rangle = \langle \lambda^\circ + \lambda^a, y\rangle \leq 0$ for all $y \in \mathcal{K}$, we must have both
\begin{equation}
\label{eq:KKT-inequalities}
\langle \lambda^\circ, y\rangle \leq 0 \quad \text{and} \quad \langle \lambda^a, y\rangle = \E[\langle \lambda, y\rangle ] \leq 0 \quad \forall y \in \mathcal{K}.
\end{equation}
To show this, we take the sequence of sets $\{F_n\}$ on which $\lambda^\circ$ is concentrated and construct the functions
\begin{equation}
\label{eq:piecewise-in-Fn}
y_n(\cdot,\omega):= \begin{cases}
y(\cdot,\omega), &\quad \omega \in F_n,\\
0, &\quad \omega \in \Omega \backslash F_n
\end{cases} \quad (y \in \mathcal{K}).
\end{equation}
Since $y \leq 0$ a.e.~$\pP$-a.s., $y_n$ also has this property, so it belongs to $\mathcal{K},$ too. We have $\langle \lambda^\circ, y_n - y \rangle = 0$ since $y_n = y$ in $F_n$ for every $n$, and moreover (since the following integrand is bounded and $\pP(F_n) \rightarrow 0$),
\[  \langle \lambda^a,y_n\rangle = \int_{F_n} \langle \lambda(\cdot, \omega), y(\cdot,\omega)\rangle \D \pP(\omega) \rightarrow 0 \quad \textup{as} \quad n \rightarrow \infty.
\]
We conclude that for any $y \in \mathcal{K},$ we have
\[
 \langle \lambda^\circ,y\rangle  = \lim_{n \rightarrow \infty} \langle \lambda^\circ + \lambda^a, y_n\rangle \leq 0.
\]
The inequality $\E[\langle \lambda, y\rangle] \leq 0$ can evidently be shown by reversing the roles of $0$ and $y$ in \eqref{eq:piecewise-in-Fn}. We have shown that \eqref{eq:KKT-inequalities} is true. 

To conclude, the first inequality is equivalent to $\lambda^\circ \in \mathcal{K}^-$; the inequality $\E[\langle \lambda, y \rangle ] \leq 0$ for any $y \in \mathcal{K}$ implies $\lambda(\cdot,\omega) \geq 0$ a.e.~in $D$ $\pP$-a.s. Due to \eqref{eq:KKT-inequalities}, the condition
\[
\langle \lambda^*, \mathcal{G}(u^*)\rangle = \langle \lambda^\circ + \lambda^a, y^* - \boldsymbol{\alpha} \rangle = 0 
\]
implies that both
\begin{equation}
\label{eq:complement-conditions-decomposition}
\langle \lambda^\circ, y^* - \boldsymbol{\alpha}  \rangle = 0 \quad \text{and} \quad \langle\lambda^a, y^* - \boldsymbol{\alpha}  \rangle = \E[\langle \lambda, y^*-\boldsymbol{\alpha} \rangle] = 0
\end{equation}
must be satisfied. 
Finally, since $\langle \lambda(\cdot,\omega),y^*(\cdot,\omega) - \boldsymbol{\alpha} \rangle \leq 0$ $\pP$-a.s., the second condition in \eqref{eq:complement-conditions-decomposition} can only be satisfied if $\langle \lambda(\cdot,\omega), y^*(\cdot,\omega)-\boldsymbol{\alpha} \rangle = 0$ $\pP$-a.s.
\end{proof}

\begin{remark}
    The arguments used in the proof of Theorem \ref{thm:OC-a.s.} greatly simplify the derivation of optimality conditions when compared to \cite{Geiersbach2022a, Rockafellar1976c}. The crucial observation used here is that one can start from the standard result (Lemma~\ref{lemma:optimality-almost-sure-basic}) and directly use the structure of the decomposition to derive the more explicit conditions. Additionally, unlike in \cite{Geiersbach2022a,Geiersbach2021c}, we did not need to require the reflexivity of $X$ in the state space $L_{\pP}^\infty(\Omega,X).$ This suggests that it may be generally possible to relax the assumption of reflexivity in other applications in PDE-constrained optimization under uncertainty.
\end{remark}
The equations \eqref{eq:KKT-operator-YH} are not yet in the form frequently used in PDE-constrained optimization. We can introduce the adjoint variable in \eqref{eq:KKT-YH-1} by defining $p(\cdot,\omega):= A^{-*}\lambda(\cdot,\omega)$ and $p^\circ:=\mathcal{A}^{-*}\lambda^\circ$, from which we obtain the adjoint equations 
\begin{equation}
\label{eq:adjoint-equation}
\mathcal{A}^* p^\circ = \lambda^\circ \quad \text{in } { \Lambda_\circ(\mathcal{C}(\bar{D})^*)} \quad \text{and} \quad A^* p = \lambda \quad \text{in } { \Lambda_1(\mathcal{C}(\bar{D})^*)}.
\end{equation}
Therefore, \eqref{eq:KKT-operator-1} can be expressed in the equivalent form
\begin{equation}
\label{eq:KKT-operator-1-alt}
 DF(u^*) + \E[p] +\mathcal{B}^*p^\circ = 0 \quad \text{in } (L^2(D))^*
\end{equation}
where $p^\circ$ and $p$ satisfy the respective conditions in  \eqref{eq:adjoint-equation}. This allows us to recover conditions analogous to \cite[Lemma 3.9]{Geiersbach2022a}:

\begin{proposition}
    \label{prop:KKT-expanded}
Under the same assumptions as Lemma~\ref{lemma:optimality-almost-sure-basic}, $u^*$ is a solution to \eqref{eq:a.s.-problem} if and only if there exist multipliers $\lambda \in { \Lambda_1(\mathcal{C}(\bar{D})^*)}$ and $\lambda^\circ \in { \Lambda_\circ(\mathcal{C}(\bar{D})^*)}$ as well as adjoint variables $p \in { \Lambda_1(L^2(D)^*)}$ and $p^\circ \in { \Lambda_\circ(L^2(D)^*)}$ such that 
\begin{subequations}
\label{eq:KKT-operator-long}
  \begin{align}
   D F(u^*) +  \E[p] +\mathcal{B}^*p^\circ &= 0, \label{eq:KKT-operator-1'}\\
  \mathcal{A}^* p^\circ - \lambda^\circ &= 0, \label{eq:KKT-operator-2'}\\
  A^*  p(\cdot, \omega) - \lambda(\cdot,\omega) & = 0, \label{eq:KKT-operator-3'}\\
  A y^*(\cdot, \omega) - u^* - \tilde{f}(\cdot, \omega) &= 0,\label{eq:KKT-operator-4'} \\
 \lambda(\cdot,\omega) \geq 0 \text{ a.e.~in D}, \quad & \langle \lambda(\cdot,\omega), y^*(\cdot,\omega)-\boldsymbol{\alpha} \rangle = 0, \label{eq:KKT-operator-5'} \\
y^*(\cdot,\omega)\leq \alpha \text{ a.e.~in D}, \quad \lambda^\circ \in  \mathcal{K}^{-}, \quad &\langle \lambda^\circ, y^*-\boldsymbol{\alpha}  \rangle = 0, \label{eq:KKT-operator-6'}
  \end{align}
\end{subequations} 
with pointwise conditions (in $\omega$) holding $\pP$-a.s.
\end{proposition}

As mentioned at the beginning of this section, the appearance of singular terms in the system \eqref{eq:KKT-operator-long} is inconvenient for numerical approximations. These are typically handled by regularization, even in the deterministic setting (since the corresponding multipliers there are generally only measures). Should additional structure (see Lemma~\ref{asrob}) be available so that the almost sure and robust formulations are equivalent, the KKT conditions derived for the latter case may be helpful in some settings for the design of numerical methods. The KKT conditions for the robust problem will be the subject of the following section.

% \begin{remark}
% As discussed before \cite[Lemma 3.16]{Geiersbach2022a}, additional regularity of the adjoint variable $p$ can be obtained since $A \colon Y \subset Y^* \rightarrow L^2(D) \subset H^{-1}(D)$.
% \end{remark}

\subsection{Robust state constraints}
\label{sec:robust-constraints}
In this section, we analyze the following robust counterpart of problem \eqref{eq:a.s.-problem}: 

\begin{subequations}
    \label{eq:robust-problem}
    \begin{alignat}{3}
    \min_{u \in L^2(D)}\, F(u) & & &&  \\
   \text{s.t.}  \quad -\Delta  \hat{y}(x,z)  &=  u(x) + f(x,z), & &&\quad (x,z) \in D \times \Xi, \label{eq:uniform-problem-constraint1} \\
 \hat{y}(x,z) &=0, & &&\quad (x,z) \in \partial D \times \Xi, \label{eq:uniform-problem-constraint2}\\
 \hat{y}(x,z) &\leq \alpha & && (x,z)\in D \times \Xi. \label{eq:uniform-problem-constraint3}
    \end{alignat}
\end{subequations}

We start with a robust version of Lemma~\ref{lem:cont-dep-sol-2}. 
\begin{lemma}
\label{lem:cont-dep-sol-robust}
Suppose Assumption~\ref{ass:PDE-standing} holds and suppose $\Xi \subset \R^m$ is a compact set.
 Then for any $u\in L^2(D)$, there exists a unique solution $\hat{y} \in \mathcal{C}(\bar{D}\times\Xi)$ of \eqref{eq:uniform-problem-constraint1}--\eqref{eq:uniform-problem-constraint2}. Moreover, there exists a constant $C>0$ such that
\begin{equation}
\label{eq:apriori-PDE-robust}
     \norm{\hat{y}}{\mathcal{C}(\bar{D}\times \Xi)}   \leq C(\max_{z \in \Xi}\norm{z}{2} + \norm{u+f_0}{L^2(D)}).
\end{equation}
\end{lemma}
\begin{proof}
From Lemma~\ref{lem:cont-dep-sol}, we have
\[
\norm{\hat{y}}{\mathcal{C}(\Xi,\mathcal{C}(\bar{D}))}=\sup_{z \in \Xi} \norm{\hat{y}(\cdot,z)}{\mathcal{C}(\bar{D})} \leq C (\sup_{z \in \Xi}\norm{z}{2} + \norm{u+f_0}{L^2(D)}) < \infty
\]
and the continuity of the mapping $z \mapsto \hat{y}(\cdot,z)$, meaning $\hat{y} \in \mathcal{C}(\Xi,\mathcal{C}(\bar{D})).$ After identifying $\mathcal{C}(\Xi, \mathcal{C}(\Bar{D}))$ with $\mathcal{C}(\bar{D}\times \Xi)$, which is justified since $\Xi$ and $\bar{D}$ are compact sets (cf.~\cite[p.~50]{Ryan2002}), we have shown \eqref{eq:apriori-PDE-robust}. 
\end{proof} 

\noindent Lemma~\ref{lem:cont-dep-sol-robust} justifies the definition of a new control-to-state operator $\hat{\mathcal{S}}\colon L^2(D) \rightarrow \mathcal{C}(\bar{D}\times \Xi)$ by
\[
\mathcal{\hat{S}}(u) = \mathcal{\hat{A}}^{-1}(\hat{\mathcal{B}}u + f),
\]
where $f \in \mathcal{C}(\Xi,L^2(D))$, $\hat{\mathcal{B}}$ is the continuous injection that maps each element of $L^2(D)$ to itself in $\mathcal{C}(\Xi, L^2(D))$, and $\hat{\mathcal{A}}^{-1}\colon \mathcal{C}(\Xi,L^2(D)) \rightarrow \mathcal{C}(\bar{D} \times \Xi)$ is defined by $(\hat{\mathcal{A}}^{-1} y)(\cdot,z) = A^{-1}y(\cdot,z).$ Analogously to the previous section, we introduce $\hat{\mathcal{G}} \colon L^2(D) \rightarrow \mathcal{C}(\bar{D}\times \Xi)$ by $[\hat{\mathcal{G}}(u)](x,z)=[\hat{\mathcal{S}}(u)](x,z)-\alpha$, the (convex and closed) cone $\hat{\mathcal{K}}:=\{y \in \mathcal{C}(\bar{D}\times \Xi) : y(x,z)\leq 0 \text{ in } D\times \Xi\}$, and the polar cone $\hat{\mathcal{K}}^{-}:=\{ y^* \in \mathcal{C}(\bar{D}\times \Xi)^*: \langle y^*,y\rangle \leq 0 \, \forall y \in \hat{\mathcal{K}}\}.$ 
\begin{lemma}
With the same assumptions as in Lemma~\ref{lem:cont-dep-sol-robust}, let $F$ be convex and continuously Fr\'echet differentiable on $L^2(D)$. Assume the constraint qualification $0 \in \textup{int} \{\hat{\mathcal{G}}(L^2(D))-\hat{\mathcal{K}} \}$ holds. Then, for $u^*$ to be a solution to \eqref{eq:robust-problem}, it is necessary and sufficient that there exists some multiplier $\lambda^* \in \mathcal{C}(\bar{D}\times \Xi)^*$ with
\begin{subequations}
\label{KKT-operator-combined}
  \begin{align}
 & \langle D F(u^*),h\rangle + \langle \lambda^*,\hat{\mathcal{A}}^{-1}\hat{\mathcal{B}}h\rangle = 0 \quad  \forall h \in L^2(D), \label{eq:KKT-operator-11}\\
 & \hat{\mathcal{G}}(u^*) \in \hat{\mathcal{K}}, \quad \lambda^* \in  \hat{\mathcal{K}}^{-}, \quad \langle \lambda^*, \hat{\mathcal{G}}(u^*) \rangle = 0. \label{eq:KKT-operator-22}
  \end{align}
\end{subequations} 
\end{lemma}

\begin{proof}
The proof follows identical arguments to those used in Lemma \ref{lemma:optimality-almost-sure-basic}.
\end{proof}

The dual space $\mathcal{C}(\bar{D}\times \Xi)^*$ can be identified with the set $\mathcal{M}(\bar{D}\times \Xi)$ of all regular Borel measures\footnote{A regular Borel measure is in this context the same as a Radon measure due to the compactness of $\Bar{D}\times \Xi$; see \cite[p.~152]{Alt2012}.} $\mu$ supported on $\bar{D}\times \Xi$ by means of
\[ 
\langle r^*, r\rangle = \int_{\bar{D}\times \Xi} r(x,z) \D \mu(x,z).
\]
This means that there exists a measure $\mu^*$ such that
\[
\langle \lambda^*, \hat{\mathcal{A}}^{-1}\hat{\mathcal{B}}h\rangle = \int_{\bar{D}\times \Xi} [ \hat{\mathcal{A}}^{-1}\hat{\mathcal{B}}h](x,z) \D \mu^*(x,z)  = \int_{\bar{D}\times \Xi} [ A^{-1}h](x) \D \mu^*(x,z).
\]
Moreover, the fact that $\lambda^*$ belongs to the set $\hat{\mathcal{K}}^-$ implies for the associated measure $\mu^*$ that $\int_{\Bar{D}\times \Xi } r(x,z)\D \mu^*(x,z) \leq 0$ for all $r \in \mathcal{C}(\bar{D}\times \Xi)$ such that $r \leq 0$ a.e.~in $\bar{D}\times \Xi$. This implies the nonnegativity of $\mu^*.$ Using $y^*:=\hat{\mathcal{S}}(u^*)$ and the definition of $\hat{\mathcal{G}}$, \eqref{KKT-operator-combined} can be refined to the conditions 
\begin{subequations}
\label{eq:KKT-operator-identif}
    \begin{align}
        &\langle D F(u^*),h \rangle  + \int_{\Bar{D}\times\Xi} [A^{-1}h](x) \D \mu^*(x,z) = 0 \quad \forall h \in L^2(D),\label{eq:KKT-operator-identif-measures1}\\
        & y^*(x,z) \leq \alpha \quad \forall (x,z)\in \Bar{D}\times\Xi, \quad \int_{\Bar{D}\times \Xi} (y^*(x,z)  -\alpha) \D \mu^*(x,z) = 0, \label{eq:KKT-operator-identif-measures2} \\
         &\mu^* \in \mathcal{M}(\Bar{D}\times\Xi) \text{ nonnegative}.\label{eq:KKT-operator-identif-measures3} 
    \end{align}
\end{subequations}

% \begin{remark}[Conditions in adjoint form]
% Introducing the adjoint variable $p^*:=-\hat{\mathcal{A}}^{-*}\lambda^*$ as in the previous section, \eqref{eq:KKT-operator-11} becomes $ D F(u^*) + \hat{\mathcal{B}}p^* =0$, where $p^*$ satisfies  $\hat{\mathcal{A}}^* p^* + \lambda^* =0$. Using the identification with the space of regular Borel measures, the optimality conditions become
% \begin{subequations}
% \label{eq:KKT-robust-expanded}
%         \begin{align}
%          D F(u^*) + \hat{\mathcal{B}}p^* &=0,\label{eq:KKT-robust-expanded1}\\
%          \hat{\mathcal{A}}^* p^* + \mu^* &=0,\label{eq:KKT-robust-expanded2} \\
%          \hat{\mathcal{A}} y^* - \hat{\mathcal{B}} u^* - f &=0, \label{eq:KKT-robust-expanded3}\\
%           y^*(x,z) \leq \alpha \quad \forall (x,z)\in \Bar{D}\times\Xi, \quad \int_{\Bar{D}\times \Xi} (y^*(x,z)  -\alpha) \D \mu^*(x,z) &= 0, \quad \mu^* \in \mathcal{M}(\Bar{D}\times\Xi) \text{ nonnegative}. \label{eq:KKT-robust-expanded4}
%     \end{align}
% \end{subequations}
% \end{remark}

\paragraph{Connection to semi-infinite programming.}
Up until now in Section~\ref{sec:almost-sure}, we have used a functional-analytic approach to derive optimality conditions. It is worth mentioning, however, that one can use arguments from semi-infinite programming to arrive at the equivalent conditions for the problem
\[
\min_{u \in L^2(D)} F(u) \quad \text{s.t.} \quad g(u):=\max_{(x,z) \in D \times \Xi} [\hat{\mathcal{S}}(u)](x,z) - \alpha \leq 0.
\]
We define the Lagrange function $\hat{L}(u,\lambda):=F(u)+\lambda g(u)$ with $\lambda \in \R.$ If there exists a $u$ such that $g(u) < 0$, then the set of Lagrange multipliers is nonempty. Then a feasible point $u^*$ is a solution if and only if there exists $\lambda^* \geq 0$ such that 
\[
0 \in \partial_u L(u^*,\lambda^*)=DF(u^*)+\lambda^*\partial g(u^*) \quad \text{and} \quad \lambda^* g(u^*) = 0.
\]
We will give a couple of equivalent characterizations of $\partial g(u^*)$. Let $g_{x,z}(u):=[\hat{\mathcal{S}}(u)](x,z)-\alpha.$
First, we consider the viewpoint of Clarke \cite[Section 2.8]{clarke1983}, where this subdifferential can be represented by the set
\[
\partial g(u^*) = \left\lbrace \int_{\bar{D}\times \Xi} \partial g_{x,z}(u^*) \D \mu(x,z) \, \Big\vert\, \mu \text{ is a Radon probability measure supported on $M(u^*)$}  \right\rbrace,
\]
and 
\[M(u):=\{ (x,z) \in \Bar{D}\times \Xi : [\hat{\mathcal{S}}(u)](x,z) = \max_{(x',z')\in \bar{D}\times \Xi} [\hat{\mathcal{S}}(u)](x',z') \}.\]
Since $\partial g_{x,z}(u) = Dg_{x,z}(u) = [\hat{\mathcal{A}}^{-1}\hat{\mathcal{B}} (\cdot )](x,z) = [A^{-1}(\cdot)](x)$, we have that there exists a Radon probability measure $\Tilde{\mu}$ supported on $M(u^*)$ such that
\begin{subequations}
\label{eq:KKT-operator-Clarke}
    \begin{align}
        &\langle DF(u^*),h\rangle + \lambda^* \int_{\Bar{D}\times\Xi} [A^{-1}h](x) \D \tilde{\mu}(x,z) = 0\quad \forall h \in L^2(D),\label{eq:KKT-operator-1''}\\
        & \lambda^* \geq 0, \quad \max_{(x,z)\in \Bar{D}\times \Xi}y^*(x,z) \leq \alpha, \quad \lambda^* \left(\max_{(x,z)\in \Bar{D}\times \Xi}y^*(x,z)  -\alpha\right) = 0, \label{eq:KKT-operator-2''}
    \end{align}
\end{subequations}

\begin{lemma}
\label{lem:equivalence-Bonnans-Clarke}
The conditions \eqref{eq:KKT-operator-identif} are  equivalent to those in \eqref{eq:KKT-operator-Clarke}. 
\end{lemma}
\begin{proof}
Let us assume we have \eqref{eq:KKT-operator-identif} and we set $\mu^*=:\lambda^*\Tilde{\mu}$ such that $\tilde{\mu}= (\mu^*(\bar{D}\times\Xi))^{-1} \mu^*$ and $\lambda^* = \mu^*(\bar{D}\times\Xi).$ Evidently, we have
$\tilde{\mu}(\bar{D}\times \Xi)=1$, making $\tilde{\mu}$ is a Radon probability measure and $\lambda^* \geq 0$ since $\mu^*$ is nonnegative. This construction immediately yields \eqref{eq:KKT-operator-1''}. Now, the integral condition in \eqref{eq:KKT-operator-identif-measures2} becomes $\lambda^* \int_{\bar{D} \times \Xi} (y^*(x,z)-\alpha) \D \tilde{\mu}(x,z) = 0.$ On account of the nonnegativity of $\tilde{\mu},$ the integrand can only be equal to zero if $y^*(x,z)=\alpha$ on a set of positive measure, and this in turn only occurs on those points satisfying $y^*(x,z) = \max_{(x',z')\in \bar{D}\times \Xi}y^*(x',z').$ In particular, this means that $\tilde{\mu}$ is supported on $M(u^*)$ and we obtain the final condition in \eqref{eq:KKT-operator-2''} due to
\begin{align*}
\lambda^* \int_{\bar{D} \times \Xi} (y^*(x,z)-\alpha) \D \tilde{\mu}(x,z) &= \lambda^* \int_{\bar{D} \times \Xi} \left(\max_{(x',z')\in \bar{D} \times \Xi}y^*(x',z')-\alpha \right) \D \tilde{\mu}(x,z) \\
&= \lambda^*\left(\max_{(x',z') \in \bar{D} \times \Xi}y^*(x',z')-\alpha\right)=0.
\end{align*}

On the other hand, suppose we have \eqref{eq:KKT-operator-Clarke}. It is clear that $\mu^*:=\lambda^*\Tilde{\mu}$ is a nonnegative Radon measure (and hence a nonnegative regular Borel measure). One can use identical arguments to those above to reconstruct the conditions \eqref{eq:KKT-operator-identif}. 
\end{proof}

We move on to another representation of the optimality conditions. The Ioffe-Tikhomirov Theorem \cite[Theorem 4.2.3]{Ioffe-Tikhomirov} yields the following representation of the convex subdifferential of $g$:
\[
\partial g(u^*)={\rm clco}\{\partial g_{x,z}(u^*)\mid (x,z)\in M(u^*)\},
\]
where ``clco'' denotes the weak closure (weak-star closure in general Banach spaces) in $U$. Recall that $\partial g_{x,z}(u^*) = [A^{-1} (\cdot )](x)$.
Then we have the optimality conditions \eqref{eq:KKT-operator-2''} with
\begin{equation}
\label{eq:KKT-operator-Ioffe-Tikhomirov}
        \langle DF(u^*),h\rangle + \lambda^* \textup{co}\lbrace [A^{-1}h](x) \mid (x,z) \in M(u^*) \rbrace \ni 0 \quad \forall h \in L^2(D).
\end{equation}
Here, ``clco'' could be replaced by ``co'' since the (finite-dimensional) set $\{ [A^{-1}h](x) \mid (x,z) \in M(u^*)\}$ is compact by the compactness of $M(u^*)$ and  $\hat{\mathcal{A}}^{-1}\hat{\mathcal{B}}h \in \mathcal{C}(\bar{D}\times \Xi)$. We have the following:
\begin{lemma}
    \label{lem:equivalence-Clarke-Ioffe-Tikhomirov}
The condition \eqref{eq:KKT-operator-Ioffe-Tikhomirov} is equivalent to the existence of a Radon probability measure $\Tilde{\mu}$ such that \eqref{eq:KKT-operator-1''} holds.
\end{lemma}

\begin{proof}
Let $S:=\{ [A^{-1}h](x) \mid (x,z) \in M(u^*)\}$. 
From \cite[Proposition 1.2]{Phelps2001} we have that a point $s$ is an element of $\mathrm{clco}(S)$ if and only if there is a Radon probability measure on $S$ (i.e., $\mu$ is supported on $M(u^*)$) representing $s$, meaning in this case that we have 
\[
f(s)=\int_{\bar{D}\times \Xi} f(x,z) \D \mu(x,z) \quad \forall f\in \mathcal{C}(\R^{d+m}).
\]
Choosing $f=A^{-1}h$, we obtain the equivalence of the formulations \eqref{eq:KKT-operator-Ioffe-Tikhomirov} and \eqref{eq:KKT-operator-1''}.
%and \cite[Proposition 3.2]{Bonnans2013}.   
\end{proof}

\bibliographystyle{abbrv}
\bibliography{references}
\end{document}